\documentclass[a4paper]{amsart}
\usepackage{amssymb,amsthm,amsmath}
\usepackage{bm}
\usepackage{esint}
\usepackage{mathrsfs}
\usepackage[backend=bibtex,style=alphabetic,bibencoding=utf8,language=auto,autolang=other,doi=false,isbn=false,url=false,eprint=false]{biblatex}
\usepackage[]{hyperref} 
\usepackage[noabbrev]{cleveref}
\newtheorem{theorem}{Theorem}[section]

\newtheorem{lemma}[theorem]{Lemma}
\newtheorem{proposition}[theorem]{Proposition}

\theoremstyle{definition}
\newtheorem{definition}[theorem]{Definition}
\newtheorem{remark}[theorem]{Remark}
\numberwithin{equation}{section}
\newcommand*{\loc}{{\text{\upshape{loc}}}}

\newcommand*{\rn}{{\mathbf{R}^n}}

\newcommand*{\R}{\mathbf{R}}

\newcommand*{\E}{\mathbf{E}}

\newcommand*{\D}{\mathscr{D}}

\newcommand*{\leb}{\mathscr{L}}

\newcommand*{\Om}{\ensuremath{\Omega}}
\newcommand*{\N}{\ensuremath{\mathbf{N}}}
\newcommand*{\F}{\mathcal{F}}
\newcommand*{\ecc}{\mathcal{E}}
\newcommand*{\M}{\mathcal{M}}
\newcommand*{\necc}{\Phi}

\newcommand*{\de}{\partial}

\DeclareMathOperator{\aff}{Aff}
\DeclareMathOperator{\Span}{span}
\DeclareMathOperator{\spt}{spt}
\DeclareMathOperator{\ran}{ran}

\DeclareMathOperator{\tr}{Tr}

\newcommand*{\A}{\mathcal{A}}
\newcommand*{\B}{\mathcal{B}}
\newcommand*{\Ahat}{\widehat{\mathcal{A}}}
\newcommand*{\Bhat}{\widehat{\mathcal{B}}}
\newcommand*{\weakstar}{\stackrel{*}{\rightharpoonup}}
\newcommand*{\areastrict}{\stackrel{E}{\rightharpoondown}}

\newcommand{\mres}{\mathbin{\vrule height 1.6ex depth 0pt width
		0.13ex\vrule height 0.13ex depth 0pt width 1.3ex}}
\newcommand \eps{\ensuremath{\varepsilon}}
\newcommand{\restricts}[2]{
	#1 
	\raisebox{-.3ex}{$|$}_{#2}
}
\DeclareMathOperator{\supp}{spt}
\DeclareMathOperator{\divergence}{div}
\author{Federico Franceschini}
\address{ETH, Rämistrasse 101, 8092 Zürich, Switzerland}
\email{federico.franceschini@math.ethz.ch}

\title{Partial regularity for $\bm{BV^\mathcal{B}}$ minimizers}

\begin{document}
\begin{abstract}
    We prove an $\eps$-regularity theorem for $BV^\B$ minimizers of strongly $\B$-quasiconvex functionals with linear growth, where $\B$ is an elliptic operator of the first order. This generalises to the $BV^\B$ setting the analogous result for $BV$ functions by F. Gmeineder and J. Kristensen [Arch. Rational Mech. Anal. 232 (2019)].

    The results of this work cannot be directly derived from the $\B =\nabla$ case essentially because of Ornstein's ``non-inequality''.

    This adaptation requires an abstract local Poincar\'e inequality and a fine Fubini-type property to avoid the use of trace theorems, which in general fail when $\B$ is elliptic.
\end{abstract}
\maketitle
\section{Introduction}\label{sec:intro}
\subsection{Main result}
In this work we prove an $\eps$-regularity theorem for $BV^\B$ - minimizers of strongly $\B$-quasiconvex functionals with linear growth, where $\B$ is an elliptic operator of the first order.

Recently (especially after \cite{DePRin16}), there has been interest to understand which results available for $BV$ maps extend to the $BV^\B$ framework (see \cite{K21,Raita17,Raita19,ArrDePRin17,ArrDePhHirRind,Arr21,Con21}). This work falls in this line of research as our main result was proved in \cite{Kristensen18} in the case $\B=\nabla$. We show that those arguments can be adapted to general first order elliptic operators.

In order to state precisely our result we introduce briefly some vocabulary, further details will be given in Section \ref{sec:preliminaries}.
\subsubsection{The operator $\B$.}\label{subsubsec:defB} We start fixing $\B$, an elliptic operator with constant coefficients, homogeneous of order 1 from $\R^{m\times n}$ to $\R^N$. That is to say, for each $v\in C^\infty(\R^n, \R^m)$ we set
\begin{equation*}
    \B v := \sum_{j=1}^n B_j\de_jv,\text{ for some linear maps }B_j\colon \R^m\to \R^N. 
\end{equation*}
By \textit{elliptic}, we mean that $\ker \Bhat [\xi] =\{0\}$ for all $\xi\in\R^n\setminus\{0\}$, where the \emph{symbol} is the linear map between $\R^m$ and $\R^N$ defined as
\begin{equation*}
    \Bhat[\xi] := \sum_{j=1}^n \xi_j B_j \text{ for each }\xi\in\R^n,
\end{equation*}
so necessarily $m\le N$.
We also define the \emph{wave cone} of $\B$ as
\begin{equation*}
    \Lambda_\B := \bigcup_{\xi\neq 0}\ran \Bhat[\xi] \ \subset \R^N.
\end{equation*}
\subsubsection{Functions with bounded $\B$-variation} This space of functions arises naturally when looking at distributional limits of sequences $\{v_k\}\subset W^{1,1}(\R^n,\R^m)$ having a bound on $\|\B v_k\|_{L^1}$. 

For an open set $\Omega\subset\R^n$, define
\begin{equation*}
    BV^\B(\Om):=\{v\in L^1(\Om,\R^m) : \B v\in\M(\Om,\R^N)\},
\end{equation*}
where $\M(\Om,\R^n)$ is the space of $\R^N$-valued Borel measures with finite total variation in $\Om$. By a famous result of Ornstein (seee \cite{Ornstein62,Kirchheim2016}), $BV^\B(\Om) \subsetneq BV(\Om,\R^m)$ unless $\B =\nabla$\footnote{That is to say, (unless) there is a linear map $p\colon \R^N \to \R^{m\times n}$ such that $\nabla u =p(\B u)$.}.

\subsubsection{Functionals defined on measures} We explain the meaning of
$$\int_\Om f(\B u)$$ 
when $u\in BV^\B$ and  $f\colon \R^N\to \R$ has linear growth
\begin{equation}\label{eq:flingrowth}
    |f(y)|\leq L\langle y\rangle\text{ for all }y\in \R^N,\tag{H1}
\end{equation}
and
\begin{equation}\label{eq:recessionfcn}
    f^\infty(y):=\lim_{t\to+\infty,y'\to y}\frac{f(ty')}{t}\text{ exists for all }y\in \Span \Lambda_\B.
\end{equation}
Here, $L>0$ and $\langle y\rangle := \sqrt{1 + |y|^2}$ is the japanese bracket. 

For all $v\in W^{1,1}(\Om,\R^m)$ consider the functional
\begin{equation}
    \F[v,\Om]:= \int_\Om f(\B v(x))\, dx,
\end{equation}
which can be extended\footnote{This extension is continuous with respect to the \textit{area-strict convergence}, see Remark \ref{rem:shortnotationmakessense}} to $BV^\B$ setting 
\begin{equation*}
    \F[v,\Om]:=\int_\Om f(\B v^{ac}(x))\, dx +\int_\Om f^\infty\Big(\frac{d\B v^s}{d|\B v^s|}\Big) \,d|\B v^s|,
\end{equation*}
where we decomposed the measure $\B v$ with respect to the Lebesgue measure. We will then denote
\begin{equation*}
    \int_\Om f(\B v)=\F[v,\Om]\text{ for all }v\in BV^\B(\Om),
\end{equation*}
(without the $dx$).
\subsubsection{$\B$-quasiconvexity}\label{subsubsec:Bqc} Following \cite{Fonseca99}, we say that a continuous function $f\colon \R^N\to\R$ is $\B$-quasiconvex if, for all $y\in\R^N$, we have
\begin{equation*}
    f(y)\leq \int_Q f(y+\B\varphi(x))\, dx\text{ for all }\varphi\in C^1_c(Q,\R^m),
\end{equation*}
where $Q\subset \R^n$ is the unit cube. 

$\B$-quasiconvex functions with linear growth are automatically Lipschitz and satisfy \eqref{eq:recessionfcn}, thus we can define $\F[v,\Om]$ for $v\in BV^\B$. Furthermore, $\F[\cdot,\Om]$ will be weakly$^*$ lower semicontinuous up to boundary terms, see Theorem \ref{thm:lsc} below.

We say that $f$ is \textit{strongly} $\B$-quasiconvex if there is $\ell>0$ such that
\begin{equation}\label{eq:fstrongBqc}
    f-\ell \langle \cdot\rangle \text{ is $\B$-quasiconvex.}\tag{H2}
\end{equation}
Strong quasiconvexity is a natural assumption in the framework of minimization problems: it is a necessary condition if we want $\mathcal{F}[\cdot,\Om]$ to be $L^1$-coercive, see \Cref{rmk:onstrongBqc}.

\subsubsection{Excess} We will prove regularity of local minimizers of $\mathcal{F}[\cdot,\Om]$ in the balls of $\Om$ where a suitable energy denisity (called the \textit{excess} following \cite{degiorgi61}) is smaller than a parameter $\eps$ which does not depend on the particular solution.

In our situation the right definition of excess in a ball $B_R(x_0)\Subset \Om$, is 
\begin{equation}\label{eq:excessintro}
    \necc(x_0,R):=\frac{1}{\omega_nR^n}\int_{B_R(x_0)}E\big(\B u -(\B u)_{B_R(x_0)}\leb^n\big),
\end{equation}
where $E(y):=\sqrt{1+|y|^2}-1$ and
\begin{equation*}
    (\B u)_{B_R(x_0)} = \frac{\B u(B_R(x_0))}{\omega_nR^n}\in \R^n.
\end{equation*}
$\Phi$ is a sort of $L^1$ oscillation of $\B u$ where we are replacing the standard norm $|\cdot|$ with $E(\cdot)$, which has the advantage of being strictly convex.

Under the further regularity assumption on the lagrangian
\begin{equation}\label{eq:freg}
    f\in C^{2,1}_\loc(\R^N),\tag{H3}
\end{equation} we are going to show 
\begin{theorem}\label{thm:epsreg}
	Let $f$ satisfy \eqref{eq:flingrowth}, \eqref{eq:fstrongBqc} and \eqref{eq:freg}, and let $u\in BV^\B(\R^n)$ be a local minimizer of $\mathcal{F}[\cdot,\Om]$, that is
    \begin{equation*}
        \int_\Om f(\B u)\leq \int_\Om f(\B u + \B \varphi) \text{ for all }\varphi\in C^1_c(\Om,\R^m).
    \end{equation*}
    Then for every $\alpha\ge 1$ and $\gamma\in(0,1)$ there is a critical threshold $\eps=\eps(\alpha,\gamma,\B,f'',L/\ell)>0$ such that the following implication holds. If
	\begin{equation*}
	 B_R(x_0)\Subset \Om,\quad	\left|(\B u)_{B_R(x_0)}\right|\leq \alpha ,\quad \necc(x_0,R)\leq \eps,
	\end{equation*}
	then $u \in C^{2,\gamma}(B_{R/2}(x_0))$ and  
\begin{equation*}
    [ \nabla^2 u ]_{C^{\gamma}(B_{R/2}(x_0))}\leq C R^{-\gamma} \sqrt{\necc(x_0,R)}.
\end{equation*} 
for some $C=C({{\alpha,\gamma,\B,f'',L/\ell}})$.
\end{theorem}
Some remarks are in order.
\begin{remark}\label{rem:existenceoff}
    Non-convex Lagrangians $f$ satisfying \eqref{eq:flingrowth}, \eqref{eq:fstrongBqc} and \eqref{eq:freg} exist, even if they are not given by explicit formulas, but rather regularizing quasi-convex envelopes.
\end{remark}
\begin{remark}\label{rem:existenceofminimizers}
    Local minimizers $u\in BV^\B_{\loc}$ to which Theorem \ref{thm:epsreg} applies, can be constructed looking at minimizing sequences of the functional $\F[v,\Om]$ in a given Dirichlet class.
\end{remark} 
\begin{remark}\label{rmk:onstrongBqc}
    Assumption \eqref{eq:fstrongBqc} is optimal in the following sense. Assume that $f$ has linear growth and $\mathcal{F}[\cdot,\Om]$ is coercive, that is for all sequences $\{v_k\}$ with fixed boudnary values on $\de \Om$ it holds:
    $$\int_\Om f(\B v_k) \text{ bounded implies } \|\B v_k \|_{L^1(\Om)} \text{ bounded.}$$
    Then necessarily there is some $\ell>0$ such that $f(\cdot)-\ell \langle\cdot\rangle $ is $\B$-quasiconvex at some $z\in\R^N$. See Section \ref{subsec:strongBqc}.
\end{remark}

\subsection{Comparison with the full gradient case}
Of course there is (a unique) $F\colon \R^{m\times n}\to \R$ such that $f(\B v)= F(\nabla v)$ and \Cref{thm:epsreg} in the case $\B = \nabla$ has been already proved in \cite{Kristensen18}.  
Still, our result cannot be reduced to the $\nabla$ case, indeed it is easily checked that $F$ satisfies \eqref{eq:flingrowth} and \eqref{eq:freg}, but \eqref{eq:fstrongBqc} does not hold for the full gradient $\nabla$. Thus we only have $L^1$ bounds $\B u$, that do not imply $L^1$ bounds on $\nabla u$, because of Ornstein non-inequality.

This fundamental difference can be worked out under the assumption that $\B$ is elliptic, which a priori was not clear at all. 

A posteriori, the main differences with respect to \cite{Kristensen18} are: a fine Fubini-type argument to bypass the lack of a trace theorem for $BV^\B$ functions (cf. \Cref{subsec:fubini}), an abstract Poincar\'e inequality to deal with $\B$-affine functions, based on the general form of Ehrenpreis fundamental principle (cf. \Cref{subsec:ehrenpreis}).

We also remark that this adaptation would be straightforward if we assumed a much stronger ellipticity condition on $\B$, namely complex ellipticity. In this case both the trace theorem and the Poincar\'e inequality are available by the results in \cite{Raita18}.

\subsection{Organization of the paper}
In \Cref{sec:preliminaries} we repeat the main definitions, fix the notation and prove the core results for $BV^\B$ functions. In \Cref{sec:proofepsreg} we prove \Cref{thm:epsreg} following the implant of \cite{Kristensen18}.

\subsection{Acknowledgements} 
The author gratefully thanks Jan Kristensen for bringing the problem to his attention and for the continuous guidance. Most of this work has been carried out while the author was visiting Oxford University. The author would like to thank the Mathematics Department and Magdalen College for the warm hospitality.

The author has been also supported by Swiss NSF Ambizione Grant PZ00P2 180042 and by the European Research Council (ERC) under the Grant Agreement No. 948029 and under the Grant Agreement No. 721675.

\section{Framework and Preliminaries}\label{sec:preliminaries}
We collect some preliminary results.
\subsection{General notation}

We work in $\R^n$ with its standard Euclidean structure, denote with $B_r$ the balls centred at the origin of radius $r>0$. $\Om$ will always denote an open, bounded set with Lipschitz boundary.

When $f\colon\R^N\to\R$ we denote differentiation by apexes
\begin{equation*}
    f'(x)[z]:=\frac{d}{dt}\Big|_{t=0}f(x+tz),\quad f''(x)[z,z]:=\frac{d^2}{dt^2}\Big|_{t=0}f(x+tz).
\end{equation*}
Similarly the action of a bilinear map $Q$ on vectors $z,z'$ is denoted by $Q[z,z']$.

We denote with $\D$ the space of test functions and with $\D'$ the space of distributions.

We denote with $C_c$ the space of continuous compactly supported functions and with $C_0$ its closure in the uniform topology.

The space $\mathcal{M}(\Om,\R^N)$ of $\R^N$-valued Borel measures on $\Om$ with finite total variation, will be identified with
	$$
	\mathcal{M}(\Om,\R^N)\simeq C_0(\Om,\R^N)^*.
	$$
	Similarly we have
	\begin{equation*}
	     \mathcal{M}_{loc}(\Om,\R^N)\simeq C_c(\Om,\R^N)^*\text{ and } \mathcal{M}(\overline\Om,\R^N)\simeq C(\overline\Om,\R^N)^*.
	\end{equation*}
	We denote with the angular bracket $\langle\cdot,\cdot\rangle$ these dualities, using the standard scalar product on $\R^N$. 

    We will use the the trace spaces $W^{s,p}$ defined by the Gagliardo seminorms
\begin{equation*}
[u]_{W^{s,p}(B_1)}:=\left(\int_{B_1}\int_{B_1}\frac{|u(x)-u(y)|^p}{|x-y|^{n+sp}}\, dx\, dy  \right)^{1/p},
\end{equation*}
we also need the sphere version 
\begin{equation*}
    [u]_{W^{s,p}(\de B_1)}:=\left(\int_{\de B_1}\int_{\de B_1}\frac{|u(x)-u(y)|^p}{|x-y|^{n-1+sp}}\, d\sigma_x\, d\sigma_y\right)^{1/p}.
\end{equation*}

In estimates we write $X\lesssim_{a,b,c} Y$ meaning that, if one fixes the parameters $a,b,c$, then the ratio $X/Y$ is bounded.
\subsection{Functionals defined on measures}
We introduce  notation to deal with functionals defined on measures. We refer to \cite{ambrosio2000functions} for background in measure theory.

\begin{definition} We say that continuous function $f\colon\overline{\Om}\times \R^N\to \R$ belongs to $\E_1(\Om,\R^N)$ if the limit
	\begin{equation*}
	\lim_{t\to +\infty}\frac{f(x,tz)}{t}=:f^\infty(x,z) \text{ exists in }\R,\text{ locally uniformly in }x\in\overline{\Om},z\in\R^N.
	\end{equation*}
The function $f^\infty(\cdot,\cdot)$ is called the ``strong recession function".
\end{definition}
We remark that, by definition, $f^\infty\colon\overline\Om\times \R^N\to \R$ is continuous and positively one-homogeneous in its second argument.

For any $f\in\E_1$ and any $\mu\in\mathcal M(\overline\Om,\R^N)$, we take the decomposition of $\mu$ with respect to the Lebesgue measure $\mu=\mu^{ac}(x)\leb^n\mres\Om+\mu^s$ and further decompose the singular part $\mu^s$ in terms of its own total variation $\mu^s=\frac{d\mu^s}{d|\mu^s|}(x)\, |\mu^s|$. Then we define for any Borel set $A\subset\overline{\Om}$
\begin{equation}\label{eq:convention}
\int_{A} f(x,\mu):=\int_{A}	 f(x,\mu^{ac}(x))\, dx +\int_{A} f^\infty\Big(x,\frac{d\mu^s}{d|\mu^s|}(x)\Big)\, d|\mu^s|(x).
\end{equation}
The same construction can of course be carried out for every Radon measure $\mu\in C_0( \Om, \R^N)^*$ and Borel set $A\Subset \Om$.

We can now define a suitable notion of strong convergence of measures.
Define $E:\R^N\to \R$ by
\begin{equation}\label{eq:areafunction}
E(z):=\sqrt{1+|z|^2}-1=\langle z \rangle-1,
\end{equation}
it is easily checked that $E$ belongs to $\E_1(\Om,\R^N)$ and it has the nice property of being strictly convex. Furthermore, simple computations shows that
\begin{equation}\label{eq:tildemu}
\int_{\overline \Om}E(\mu)=|\tilde \mu|(\overline{\Om})\quad \text{ where }\tilde \mu:=\left(\mu,\leb^n\mres\Om\right)\in C(\overline{\Om},\R^N\times \R)^*.    
\end{equation}

\begin{definition}[Area-strict convergence] Given $\mu$ and $\{\mu_j\}_{j\in\N}$ in $\mathcal M(\Omega,\R^N)$ we say that $\{\mu_j\}_{j\in\N}$ converges ``area-strictly" to $\mu$ in $\Om$, and write $\mu_j \areastrict \mu$, as $j\to +\infty$, provided $\mu_j\weakstar \mu$ in $C_c(\Om,\R^N)^*$ and
\begin{equation}\label{eq:areastrictconv}
    \int_{ \Om}E(\mu_j)\to \int_{ \Om}E(\mu)  \quad \text{ as }j\to +\infty.
\end{equation}
\end{definition}


Intuitively \eqref{eq:areastrictconv} prevents oscillations and loss of mass to $\de \Omega$, this ensures continuity of the functional $\mu \mapsto \int_{\Om}f(x,\mu)$ for all $f\in \E_1(\Om,\R^N)$ as the following version of Reshetnyak continuity shows.
\begin{theorem}[Theorem 5 in \cite{Kristensen2009}]\label{thm:areacont}
	For every $f\in \E_1(\Om,\R^N)$ we have 
	$$
	\lim_{j\to +\infty}\int_{\overline \Om}f(x,\mu_j)=\int_{\overline\Om} f(x,\mu),
	$$
	provided $\mu_j\areastrict \mu$ in $\Om$.
\end{theorem}
\begin{remark}\label{rmk:areacontimproved} If we know something more about the limit measure $\mu$, we can relax the assumptions on $f$. In fact, what is really needed in the proof of \Cref{thm:areacont} is that the ``perspective integrand''
$F\colon\overline\Om\times \R^N\times \R \to \R$, defined by 
\begin{equation*}
F(x,z,t):=\begin{cases}|t|\, f(x,z/|t|)& \text{ if } t\neq 0,\\
f^\infty(x,z)&\text{ if } t=0,\end{cases}
\end{equation*} has a $|\tilde\mu|$-neglegible set of discontinuity points (see \cite[Proposition 1.62, (b)]{ambrosio2000functions}), where $\tilde \mu$ is as in \eqref{eq:tildemu}.
\end{remark}
\begin{remark}\label{rem:shortnotationmakessense}
    Definition \eqref{eq:convention} can be justified a posteriori by Theorem \ref{thm:areacont}. In fact \eqref{eq:convention} is obtained as extension by area strict continuity of 
$$
\phi\mapsto\ \int_{\Om} f(x,\phi(x))\, dx,
$$
where we think $\phi\in L^1(\Om,\R^N)\subset \mathcal{M}(\Om,\R^N)$.
\end{remark}
\subsection{The operators $\A$ and $\B$}\label{subsec:operators}

We fix a homogeneous, first-order, elliptic differential operator $\B$ with constant coefficients over $\rn$ from $\R^m$ to $\R^N$, as explained in \Cref{subsubsec:defB}.

The Leibniz rule takes the form
\begin{equation}\label{eq:Bleibniz}
 \B(\eta u) = \eta \B u + \Bhat[\nabla\eta]u,\text{ for all }\eta\in C^\infty(\R^n),u\in C^\infty(\R^n,\R^m).
\end{equation}
Exploiting the ellipticity assumption we find another finite dimensional vector space $\R^d$ and a \emph{homogeneous} differential operator $\A$ over $\rn$ from $\R^N$ to $\R^d$ such that
\begin{equation}
\tag{symbol exactness}
\ran \Bhat[\xi]=\ker \Ahat[\xi]\qquad 	\text{ for all }\xi\in \rn\setminus\{0\}.
\end{equation}
Notice that $\A$ might have order larger than one. The existence of such a couple $(d,\A)$ is not obvious nor unique, (see \cite[Proposition 4.2]{vs}). In the elliptic case one can, for example, set $d:=N$ and define $\A$ via its symbol
$$
\Ahat:={\det\left(\, \Bhat^\dagger\circ\Bhat\, \right)}\cdot \left\{\text{id}_{\R^N}-\Bhat \circ {\left(\, \Bhat^\dagger\circ\Bhat\, \right)}^{-1}\circ \Bhat^\dagger\right\},
$$
homogeneity and symbol exactness are simple to check. 

Finally, we remark that $\B u(x)\in \Span \Lambda_\B$, so we can (and do) always restrict ourselves to the case $\R^N =\Span \Lambda_\B$.

\subsection{Erhenpreis fundamental principle}\label{subsec:ehrenpreis}
We state a very general result concerning the solvability of (possibly overdetermined) systems of PDEs with constant coefficients.

Consider the set $M_\B$ of all constant-coefficients differential operators 
\[\alpha\colon C^\infty(\rn,\R^m)\to C^\infty(\rn)\] 
such that $\alpha\circ \B \equiv 0$. We will call $M_\B$ the ``module of compatibility conditions'' of $\B$. The following remarkable result is contained in \cite[Theorem 1, Chapter 7]{Palamodov70}.

\begin{theorem}[]\label{palamodov}
	Let $V\subset \rn$ be an open convex set and ${M}_\B$ its module of compatibility conditions. If $f\in\D'(V,\R^m)$ satisfies 
    \begin{equation}\label{eq:compatibility}
    \alpha(f)=0\text{ in }\D'(V)\text{ for every }\alpha\in{M}_\B,
    \end{equation}
	then there exists $u\in\D'(V,X)$ such that $\B u=f$ in $\D'(V,\R^m)$.
\end{theorem}
\begin{remark}
    Then ${M}_\B$ has a natural structure of $\R[\xi_1,\ldots,\xi_n]$-module via the natural action:
$$
p\cdot \alpha :=p(\de/\de x_1,\ldots ,\de/\de x_n)\circ \alpha,\qquad \text{ for all } p\in\R[\xi],\alpha\in{M}_\B.
$$
Since ${M}_\B$ is a submodule of the Noetherian module $\R[\xi_1,\ldots,\xi_n]$, it is finitely generated by some operators $\{\A_1,\ldots,\A_\ell\}$. This means that the compatibility condition \eqref{eq:compatibility} in this theorem can be checked only for $\A_1,\ldots,\A_\ell$.
\end{remark}
\begin{remark}
This theorem is a far reaching generalization of Poincar\'e's Lemma: a closed form (that is a form that satisfies the compatibility conditions $d\omega=0$) is in fact exact, provided the domain is simple enough (for example convex always works). 
A quicker account of this theory can be found in \cite[Chapter 7]{Hormander90}. 
\end{remark}

\subsection{The space $BV^\B$} Let us collect some properties of the space $BV^\B$.
Recall that 
\begin{equation*}
BV^\B(\Om):=\left\{u\in L^1(\Om,\R^m)\colon \B u\in C_0(\Om,\R^N)^*\right\}.
\end{equation*}
The space $BV^\B_\loc(\Om)$ is defined similarly requiring that $\B u\in C_c(\Om,\R^N)^*$. 

Given some sequence $\{u_j\}_{j\in\N}\subset BV^\B(\Om)$ and $u\in BV^\B(\Om)$, we say that:
\begin{itemize}
	\item $\{u_j\}$ converges weakly$^*$ to $u$ in $BV^\B_\loc(\Om)$ provided
	$$
	u_j\to u\text{ in }L^1_{\loc}(\Om)\quad \text{ and }\quad \B u_j\weakstar \B u\text{ in }C_c(\Om,\R^N)^*;
	$$
	\item $\{u_j\}$ converges area-strictly to $u$ in $BV^\B( \Om)$ provided $\lim_j\|u-u_j\|_{L^1(\Om)}=0$ and
	$$
	\int_{\Om}E(\B u_j) \to \int_{\Om}E(\B u),\quad \text{ that is }\B u_j\areastrict \B u\text{ in }\Om.
	$$
\end{itemize}
The area-strict closure of $C^\infty_c(\Om,\R^m)$ is denoted by $BV^\B_0(\Om)$.

$BV^\B$ functions can be approximated by smooth ones in the following sense.
\begin{lemma}\label{lem:approximationBVB}
	Let $u\in BV^\B_{\loc}(\Om)$ and $\{\phi_\varepsilon\}_{\eps>0}$ be a family of standard mollifiers, extend $u$ to zero outside $\Om$ and set $u_\varepsilon:=u*\phi_\varepsilon$. Then $u_\eps$ converges weakly* to $u$ in $BV^\B_\loc(\Om)$.
    Furthermore, for every $U\Subset \Om$ such that $\leb^n(\de U)+|\B u|(\de U)=0$, convergence holds in the area-strict sense in $BV^\B(\omega)$.
\end{lemma}

The following embedding is crucial for the proof to work in the $\B\ne \nabla$ case
\begin{lemma}\label{lem:embeddingsBVB}
    For all $u\in BV^\B (B_1)$ and $p\in(1,\frac{n+1}{n})$ we have
    \begin{equation}\label{eq:Wspembedding}
        \| u \|_{W^{1-1/p,p}(B_{1/2})} \lesssim_{\B,p} |\B u|(B_1) +\|u\|_{L^1(B_1)}.
    \end{equation}
\end{lemma}
\begin{proof}
    As the inequality is local, we can reduce ourselves, by multiplying by a cutoff function, to prove \eqref{eq:Wspembedding} in the whole $\R^n$. Furthermore, thanks to Lemma \ref{lem:approximationBVB}, we can assume $u\in C^1_c(\R^n,\R^m)$.
    
    Let $s<0$ to be fixed later. First, recall that by Sobolev  embedding it holds 
    \begin{equation*}
         H^{-s,p'}(\R^n)\subset C_0(\R^n)\text{ provided } -sp'>n,
    \end{equation*}
    where $p'$ is the H\"older conjugate. Taking duals we find $\mathcal{M}(\R^n)\subset H^{s,p}(\R^n)$.
    Mihlin multiplier theorem $(p>1)$ and ellipticity of $\B$ then give
    \begin{equation*}
        \|\nabla u\|_{H^{s,p}}\lesssim_{n,s,p} \|\B u \|_{H^{s,p}}.        \end{equation*}
        Furthermore, in this range of $p$\footnote{Since $\B^{-1}$ can be represented as a convolution kernel which is $1-n$ homogeneous, by classical boundedness of convolution operators we have $BV^\B(\R^n)\subset L^{n/(n-1)}_{weak}(\R^n)$. Since $p<n/(n-1)$, we deduce $BV^\B(\R^n)\subset L^p(\R^n)$.} we have
        \begin{equation*}
            \|u\|_{L^p(\R^n)}\lesssim_{n,p} \|\B u\|_{L^1(\R^n)} + \|u\|_{L^1(\R^n)}.
        \end{equation*}
        Thus, we proved that $BV^\B(\R^n)\subset H^{1+s,p}$, provided $sp'>n$.
 
        Now we conclude using the relationship between Hardy and Besov spaces, we refer to \cite{berghlofstrom76} for background. 
        
        Since $p\in(1,2]$ we have
        \begin{equation*}
            H^{1+s,p}(\R^n) \subset B^{1+s}_{p2}(\R^n),
        \end{equation*}
        see for example \cite[Theorem 6.4.4]{berghlofstrom76}. Then, for all $\delta>0$, we use that (see for example \cite[Theorem 5.2.1]{Triebel73})
        \begin{equation*}
            B^{1+s}_{p2}(\R^n)\subset B^{1+s-\delta}_{pp}(\R^n).
        \end{equation*}
        Since $B^{1-1/p}_{pp}(\R^n) = W^{1-1/p,p}(\R^n)$, we are finished if we find $s,\delta$ such that
        \begin{equation*}
           \delta>0,\quad s<0, \quad 1+s-\delta =1/p', \quad -sp'>n. 
        \end{equation*}
        Such choice is possible if and only if $1/p>n/p'$, which rewrites as $p<(n+1)/n$.
\end{proof}
\begin{remark}
    Recall that $W^{1-1/p,p}(\R^{n-1})$ is the trace space of $W^{1,p}(\R^n)$, for all $p>1$, thanks to Gagliardo's trace theorem \cite{Gagliardo57}.
\end{remark}
We define the set of $\B$-affine maps in an open set $U\subset\R^n$ as
\begin{equation*}
    \aff(\B,U) = \{u\in \D'(U,\R^m) : \B u \equiv\text{cost.} \} = \{ u\in C^\infty(U,\R^m) : \B u \equiv\text{cost.}\},
\end{equation*}
and the kernel
\begin{equation*}
    \ker(\B,U) = \{u\in \D'(U,\R^m) : \B u = 0 \} = \{ u\in C^\infty(U,\R^m) : \B u = 0 \},
\end{equation*}
where we use the local regularity of constant coefficients elliptic operators. These maps in general depend on $U$.

The closed graph theorem then entails the local Poincar\'e inequality.
\begin{proposition}[Poincaré]\label{prop:weakpoincare}
	For all $u\in BV^\B(B_1)$ we have
	\begin{equation}\label{eq:poicnareL1}
	\inf_{h\in \ker(\B,B_1)}\|u-h\|_{L^1(B_{1/2})}\lesssim_{\B}|\B u|(B_1).
	\end{equation}
	And, for all $p\in(1,\frac{n+1}{n})$ and $R>0$
\begin{equation}\label{eq:poincareWsp}
	\inf_{h\in \ker(\B,B_R)}[u-h]_{W^{1-1/p,p}(B_{R/2})}\lesssim_{\B,p}R^{\frac{n+1}{p}-n}|\B u|(B_R).
\end{equation}
\end{proposition}
\begin{proof}
    We start reformulating \eqref{eq:poicnareL1} abstractly.  Consider the Frech\'et spaces
    \begin{equation*}
        X:=BV^\B_{\loc}(B_1)\text{ and }Y:=\mathcal{M}_{\loc}(B_1,\R^N).
    \end{equation*}
    The structure of Frech\'et spaces is induced by the seminorms ($k\ge 0$)
    \begin{equation*}
        p^X_k(u):=|\B u|(B_{1-2^{-k}}) +\int_{B_{1-2^{-k}}}|u|\quad\text{ and }\quad p^Y_k(\mu)=|\mu|(B_{1-2^{-k}}).
    \end{equation*}
    If we consider the continuous map $\B\colon X\to Y$ sending $u$ to $\B u$, \eqref{eq:poicnareL1} is equivalent to show that the inverse of the map 
    \begin{equation*}
        \widetilde\B\colon X/\ker(\B,B_1)\to \ran \B
    \end{equation*}
    is continuous.
    Thus, by the open mapping theorem, everything boils down to prove that $\ran \B$ is closed in $Y$. Consider a sequence $u_j\in X$ such that $\B u_j \to \mu$ in $Y$ for some measure $\mu$. Then for all $\alpha\in M_\B$ (see Section \ref{subsec:ehrenpreis}) we have
    \begin{equation*}
        \alpha(\mu)=\D'-\lim_j \alpha(\B u_j) =0.
    \end{equation*}
    Then Theorem \ref{palamodov} applies and gives $v\in \D'(B_1,\R^m)$ such that $\B v=\mu$. Finally, since $\B$ is elliptic, $\B u\in \mathcal{M}_{\loc}$ forces $v\in L^q_{\loc}$ for all $q\in[1,n/n-1)$, thus $v\in X$ and $\mu\in\ran \B$.

    We turn to the proof of \eqref{eq:poincareWsp}, by scaling we can take $R=1$. Take a smooth cutoff function $\bm 1_{B_{1/2}}\leq \varrho\leq \bm 1_{B_{2/3}}$ and any $h\in \ker(\B,B_{1})$, then using Lemma \ref{lem:embeddingsBVB} we have
	\begin{align*}
	[u-h]_{W^{1-1/p,p}(B_{1/2})}&\leq[\varrho(u-h)]_{W^{1-1/p,p}(\R^n)}\\
    &\lesssim_{n,s,\B}\left\|\B(\varrho(u-h))\right\|_{\mathcal M(\R^n)}+\|\varrho(u-h)\|_{L^1(\R^n)}\\
	&\leq \left\|\varrho \B u\right\|_{\mathcal M(\R^n)}+\left\|\widehat{\B}[\nabla\varrho](u-h)\right\|_{L^1(\R^n)}+\|\varrho(u-h)\|_{L^1(\R^n)}\\
	&\lesssim_\B |\B u|(B_1)+ \|u-h\|_{L^1(B_{3/2})}.
	\end{align*}
    Now thanks to \eqref{eq:poicnareL1} we can choose $h$ so that $\|u-h\|_{L^1(B_{2/3})}\lesssim |\B u|(B_1)$ and we are done. 
\end{proof}
For $u\in BV^\B_{\loc}(\Om)$ consider the Lebesgue decomposition of the measure $\B u$
$$
\B u= (\B u)^{ac}\, \leb^n\mres \Om + \frac{d(\B u)^s}{d|\B u|^s}\, |\B u|^s,
$$
where the Borel function $ \frac{d(\B u)^s}{d|\B u|^s}\colon \Om\to \R^N$ is defined only $|\B u|^s$-almost everywhere. Then the following generalization of Alberti's rank-one Theorem (\cite{alberti93,DePRin16}) holds
\begin{equation}\label{eq:rankone}
\frac{d(\B u)^s}{d|\B u|^s}(x) \text{ belongs to } \Lambda_\B\text{  for $|\B u|^s$-a.e. $x\in \Om$}.
\end{equation}

\subsection{Fubini property}\label{subsec:fubini}
We will also use a Fubini-type property for maps $f\in W^{s,p}(B_1,\R^m)$. 
\begin{lemma}\label{lem:goodradius}
    Let $s\in(0,1),p>1$ and let $\phi$ be a standard mollifier. Let $f\in W^{s,p}(B_1,\R^m)$ and denote by $B_1\setminus S_f$ the set of $L^p$-Lebesgue points of $f$ and by $\tilde f\colon B_1\setminus S_f \to \R^m$ the precise representative. 
    
    For all $0<r<R<1$ with $R\le 100 r$, there is a set of ``good radii'' $G\subset(r,R)$ such that $\leb^1(G)>0$ and for all $t\in G$ the following holds
    \begin{itemize}
        \item[(a)] $\mathcal H^{n-1}(S_f\cap \de B_t)=0$;
        \item[(b)] for $\phi_\eps:=\eps^{-n}\phi(\cdot/\eps)$ it holds
        \begin{equation*}
            \|\tilde f -f * \phi_\eps \|_{L^p(\de B_t)} \to 0\text{ as } \eps \downarrow 0;
        \end{equation*}
        \item[(c)] we have the bound
        \begin{equation*}
        [\tilde f]_{W^{s,p}(\de B_t)} \lesssim_{n,s,p,}(R-r)^{-1/p}[f]_{W^{s,p}( B_R)}.
    \end{equation*}
    \end{itemize}
\end{lemma}
\begin{proof}
    Assume that $[f]_{W^{s,p}(B_1)}\le1$, the proof is based on the Fubini inequality   \begin{equation}\label{eq:fgh}
        \int_r^R [f]_{W^{s,p}(\de B_t)}dt \lesssim 1.
    \end{equation} 
    While known (see \cite[Proposition 8.25]{gmeinederox}), let us sketch its proof.
    Recall the following norm which is equivalent \footnote{This is based on the nice fact that for subadditive functions $0\le g(x+x')\le g(x)+g(x')$ we have $\sup_{B_1} g \lesssim \int_{B_1} g(x)/|x|^n$.} to the $W^{s,p}$ one (see for example \cite[Proposition 17.21]{leoni2017}):
    \begin{equation*}
        |f|^p_{W^{s,p}(\R^n)}:=\int_{0}^\infty \sup_{\xi\in \R^n,|\xi|\le t}\|f-f(\cdot-\xi)\|^p_{L^p(\R^n)}\frac{dt}{t^{1+sp}}.
    \end{equation*}
    Now for $f\in C^1_c(\R^n,\R^m)$ we have
    \begin{align*}
        \int_0^\infty[f(\cdot,s)]^p_{W^{s,p}(\R^{n-1})}ds&=\int_0^\infty ds \int_{\R^{n-1}}dx'\int_{\R^{n-1}}dy' \frac{|f(x',s)-f(y',s)|^p}{|x'-y'|^{n-1+sp}}\\
        (\tau':=x'-y')\quad &= \int_{\R^{n-1}}\frac{d\tau'}{|\tau'|^{n-1+sp}}\|f(\cdot-(\tau',0))-f\|^p_{L^p(\R^{n-1}\times(0,\infty))}\\
        &\le \int_{\R^{n-1}}\frac{d\tau'}{|\tau'|^{n-1+sp}}\sup_{\xi\in\R^n,|\xi|\le |\tau'|}\|f(\cdot-\xi)-f\|^p_{L^p(\R^{n})}\\
        (\text{polar})\quad& =\int_0^\infty\frac{dt}{t^{1+sp}}\sup_{\xi\in\R^n,|\xi|\le t}\|f(\cdot-\xi)-f\|^p_{L^p(\R^{n})}\\
        &=|f|^p_{W^{s,p}(\R^n)}\le C(n,s,p)[f]^p_{W^{s,p}(\R^n)}.
    \end{align*}
    This computation works also slicing with spheres, at least as long as $r/R$ is bounded below.

    Now (a) is just a consequence of Fubini's theorem in polar coordinates and Lebesgue differentiation Theorem, indeed we have $\leb^1(I)=R-r$ where
    \begin{equation*}
        I:=\{t\in[r,R] : \mathcal{H}^{n-1}(S_f\cap \de B_t) =0\}.
    \end{equation*}
    
    In order to prove (b), we apply \eqref{eq:fgh} to $f_\eps=f*\phi_\eps$, and employ Fatou's Lemma to find $\leb^1(J)=R-r$ where
    \begin{equation*}
        J:=\{ t \in [r,R] : \liminf_\eps \|f_\eps\|_{W^{s,p}(\de B_t)}<\infty\}.
    \end{equation*}
    So, by Rellich's Theorem, for each $t\in J$ we have that $\{f_{\eps}|_{\de B_t}\}_{\eps}$ is pre-compact in $L^p(\de B_t)$. If $t\in I\cap J$, then necessarily 
    \begin{equation*}
        f_\eps \to \tilde f \text{ in } L^p(\de B_t),
    \end{equation*}
    by uniqueness of the $\mathcal H^{n-1}$-a.e. limit.
    
    Finally (c) follows by \eqref{eq:fgh} and the mean value inequality.
\end{proof}

\subsection{$\B$-quasiconvexity and lower semicontinuity}

Let us first repeat the definition given in \Cref{subsubsec:Bqc}:
\begin{definition}[$\B$-quasiconvexity] 
	A locally bounded Borel function $f\colon \R^N \to \R$ is said to be $\B$-quasiconvex if, for all $y\in \R^N$, it holds
	\begin{equation}\label{eq:Bqcdefinition}
	f(y)=\inf\left\{\fint_{Q} f(y+\B \varphi(x))\, dx: \varphi \in \D(Q,\R^m)\right\},
	\end{equation}
	where $Q$ is the unit cube in $\R^n$. 
\end{definition}
\begin{remark}
    This is equivalent to ask that $f$ is $\A$-quasiconvex in the sense of \cite{Fonseca99}, see for example \cite[Corollary 6, Lemma 5]{Raita18}.
\end{remark}
It is immediate to check that \eqref{eq:Bqcdefinition} is equivalent to require that $f\circ \beta\colon \R^m\otimes \R^n \to \R$ is quasi-convex in the sense of Morrey (\cite{Morrey52}), where $\beta\colon \R^m\otimes \R^n \to \R^N$ is the linear map such that $\B u(x)=\beta(\nabla u(x))$, which is given by
\begin{equation*}
    \beta(v\otimes \xi):= \Bhat[\xi] v\text{ for }\xi\in\R^n, v\in \R^m.
\end{equation*}
Since $\R^N=\Span\Lambda^\B$, $\beta$ is surjective, but in general not injective; nevertheless $\Lambda^\B =\beta(\Lambda^\nabla)$ since $\Lambda^\nabla=\{\text{rank-one matrices}\}$. This observation grants that many properties of $f$ are immediately deduced from the corresponding properties of Morrey's quasiconvex functions, for example
\begin{lemma}
    If $f$ is $\B$-quasiconvex and has linear growth, then it is globally Lipschitz and $\Lambda_\B$-convex, meaning that for all $t\in[0,1],y,y'\in\R^N$ it holds
    \begin{equation*}
        f(ty+(1-t)y')\le tf(y) +(1-t) f(y'),\text{ provided } y-y'\in \Lambda_\B.
    \end{equation*}
\end{lemma} 
This Lemma entails the following
\begin{proposition}	\label{prop:Bqcconsequences}
	Assume $f\colon \R^N\to \R$ is $\B$-quasiconvex and has linear growth.
	Consider the upper and lower recession functions
	\begin{equation*}
	    f^\#(y) =\limsup_{y'\to y,t\uparrow \infty} f(ty')/t,\quad f_\#(y) =\liminf_{y'\to y,t\uparrow \infty} f(ty')/t,
	\end{equation*}
	which are real valued and positively 1-homogeneous. Then $\Lambda_\B\subset \{f^\# = f_\#\}$ so that the functional $u\mapsto \int_\Om f(\B u)$ is well defined on $BV^\B(\Om)$ and area-strict continuous (cf. \Cref{rmk:areacontimproved} and \eqref{eq:rankone}).
\end{proposition}
Now that we known how to give a meaning to $\int_\Om f(\B u)$ it is natural to hope that this functional is l.s.c. with respect the weak* convergence (this is not used in the proof of \Cref{thm:epsreg}). The only issue is the possible concentration of mass on $\de \Om$, which is a problem since no trace operator is available for a general elliptic operator in the $L^1$ setting \footnote{This would no be an issue if $\B$ was complex elliptic.}. We report the following
\begin{theorem}\label{thm:lsc}
    Let $u,\{u_j\}_{j\in\N}$ in $BV^\B(\Om)$ and assume $u_j\weakstar u$. Then there exist a measure $\lambda\in\mathcal M(\overline\Om)^+$ such that
    \begin{equation*}
        \liminf_j \int_\omega f(\B u_j)\geq \int_\omega f(\B u) 
    \end{equation*}
    whenever $\omega\subset \Om$ is an open set such that $\lambda(\de \omega) =0$.
\end{theorem}
\subsection{On the strong $\B$-quasiconvexity assumption}\label{subsec:strongBqc}
The following striking result fully justifies assumption \cref{eq:fstrongBqc}, it follows with notational changes from the case $\B =\nabla$, see \cite[Proposition 3.1]{Kristensen18} and \cite{Chen17} for a more detailed treatment. \begin{proposition}    Assume $f\colon \R^N\to \R$ is a continuous integrand of linear growth, let $\Om\subset \R^n$ be a bounded Lipschitz domain and $g \in W^{1,1}(\R^n, \R^N )$. Then minimizing sequences for the variational problem
\begin{equation*}
    \inf_{u\in W^{1,1}(\R^n,\R^N),\supp(u-g)\subset \Om }\int_\Om f(\B u(x))\, dx
\end{equation*}
are all bounded in $BV^\B(\Om)$ if and only if $f-\ell E$ is $\B$-quasiconvex at some point $y_0\in\R^N$, for some $\ell >0$.
\end{proposition}

\subsection{$L^p$ regularity for Legendre-Hadamard elliptic systems}
A symmetric bilinear form $Q\colon \R^N\times \R^N \to R$ is called $\B$-Legendre--Hadamard elliptic if there is $\lambda >0$ 
\begin{equation}\label{eq:blegendrehadamarddef}
    Q(\Bhat[\xi]v,\Bhat[\xi]v)\geq \lambda |\xi|^2|v|^2\text{ for all }\xi\in\R^n,v\in \R^m,
\end{equation}
the positive constant $\lambda$ is called the ellipticity constant.
This is equivalent to ask that $\widetilde Q = Q(\beta(\cdot),\beta(\cdot))$ is Legendre--Hadamard elliptic on $\R^m\otimes\R^n$. When $f-\ell E $ is $C^2$ and $\B$-quasiconvex, using that $f\circ\beta$ is rank-one convex we immediately find that $f''(y)$ is $\B$-Legendre--Hadamard elliptic for each $y$ for some $\lambda=\lambda(\ell,\B,y)$.
\begin{theorem}\label{ellipticregconstantcoeff} Let $Q\colon \R^N\times \R^N\to \R$ be a symmetric, $\B$-Legendre--Hadamard elliptic, bilinear form with ellipticity constant $\lambda$ and $|Q|\leq \Lambda$. Given some ball $B=B_R(x_0)\subset \rn$ and some exponents $p\in(1,+\infty)$ and $q\geq 2$, the following holds.
	\begin{itemize}
		\item[(a)] For each $g\in W^{1-1/p,p}(\de B,\R^m)$ there exist a unique solution $h\in W^{1,p}(B,\R^m)$ to the system
		\begin{equation}\label{eq:systema}
		\begin{cases}-\B^*\left(Q.\B h\right)=0&\text{ in }B,\\
		\tr_{\de B}(h)=g&\text{ on }\de B,\end{cases}
		\end{equation}
		where the first equation is intended in the distribution sense, and
		\begin{equation*}
		\|\nabla h\|_{L^p(B,\R^m\otimes \rn)}\lesssim_{n,p,\B,\Lambda/\lambda}[g]_{W^{1-1/p,p}(\de B,\R^m)}.
		\end{equation*}
		Furthermore, $h\in C^\infty(B,\R^m)$ and for every $0<r<R$ and $z\in \R^m\otimes \rn$:
		\begin{equation*}
		\sup_{B_{r/2}(x_0)} |\nabla h-z|+r\, \sup_{B_{r/2}(x_0)}|\nabla^2 h|\lesssim_{n,\B,\Lambda/\lambda}\fint_{B_r(x_0)}|\nabla h(x)- z|\, dx.
		\end{equation*}
		\item[(b)] For each $f\in L^q(B,\R^m)$ there exist a unique solution $w\in W^{2,q}(B,\R^m)$ to the system
		\begin{equation}\label{eq:systemb}
		\begin{cases}-\B^*\left(Q.\B w\right)=f&\text{ in }B,\\
		\tr_{\de B}(w)=0&\text{ on }\de B,\end{cases}
		\end{equation}
		where the first equation is intended in the distribution sense, and
		\begin{equation*}
		\| w\|_{W^{2,q}(B,\R^m\otimes \rn)}\lesssim_{n,q,\B,\Lambda/\lambda,R}\|f\|_{L^q( B,\R^m)}.
		\end{equation*}
	\end{itemize} 
\end{theorem}
\begin{proof}
	We just show how to reduce ourselves to the case $\B=\nabla$, then the theorem is essentially the $L^p$ regularity of Legendre-Hadamard elliptic systems, see for example \cite[Proposition 2.11]{Kristensen18} and the references therein. 
    
    Suppose you have a solution $h\in W^{1,p}(B,\R^m)$ of the system \eqref{eq:systema}, then
	$$
	0=\int_B Q[\B h(x),\B\varphi(x)]\, dx=\int_B \widetilde Q[\nabla h(x),\nabla\varphi(x)]\, dx.
	$$ 
	Conversely, the same formula shows that if we are able to solve the system
	\begin{equation*}
	\begin{cases}-\divergence\big(\widetilde Q.\nabla h\big)=0&\text{ in }B,\\
	\tr_{\de B}(h)=g&\text{ on }\de B,\end{cases}
	\end{equation*}
	for any Legendre-Hadamard elliptic form $\tilde Q$ with ellipticity constants $\sim \lambda,\sim\Lambda$, then we have in fact found a solution of system \eqref{eq:systema}.
	An identical reasoning applies to system \eqref{eq:systemb}. 
\end{proof}
\subsection{Some auxiliary estimates for \texorpdfstring{$\bm{E}$}{} and \texorpdfstring{$\bm{f}$}{}} 
In this section we collect some auxiliary estimates, we start with the area function $E$.
\begin{lemma}\label{lem:auxiliaryboundsE}
	For every $y,y'\in \R^N$ and $\alpha \geq 1$:
\begin{align}
&(\sqrt{2}-1)\, \min\left\{|y|^2,|y|\right \}\leq E(y)\leq \min\left\{|y|^2,|y|\right \},\label{Emin}\\
&E(\alpha y)\leq \alpha^2\, E(y),\label{quasihomogeneity}\\
&E(y+y')\leq 2\left( E(y)+E(y')\right).\label{quasitriangular}
\end{align}
\end{lemma}
\begin{lemma}[Lemma 2.8 in \cite{Kristensen18}]\label{quasiminimality}
	For every $\mu\in C_0(\omega, \R^N)^*$ we have
	\begin{equation}
	\inf_{y\in \R^N}\int_\omega E(\mu - y)\leq \int_\omega E(\mu -(\mu)_\omega)\leq 4\, \inf_{y\in \R^N}\int_\omega E(\mu - y),
	\end{equation}
where $(\mu)_\omega:=\frac{\mu(\omega)}{\leb^n(\omega)}$ is the mean value of $\mu$ and $\omega \subset \R^n$ is an open set.
\end{lemma}
\begin{lemma}[Lemma 2.9 in \cite{Kristensen18}]\label{sqaureroot}
For every $\mu\in C_0(\omega, \R^N)^*$ we have
\begin{equation}
\frac{1}{\leb^n(\omega)}\int_\omega |\mu|\leq 	\sqrt{\necc^2+2\necc}\quad \text{ with }\necc:=\frac{1}{\leb^n(\omega)}\int_\omega E(\mu).
\end{equation}
In particular for $\necc\leq 1$ we have $\fint_\omega |\mu|\leq\sqrt{3\necc}$.
\end{lemma}
It's clear that $E$ is strictly convex, the following bounds explicitly show that the ellipticity constants are bounded below on every compact set.
\begin{lemma}[Lemma 4.1 in \cite{Kristensen18}]\label{auxiliary1} For all $y_0,y\in \R^N$ we have
	\begin{align*}
	& E''(y_0)[y,y]=\Big\{1+|y_0|^2-|y_0|^2\big(\frac{y_0}{|y_0|}\cdot\frac{y}{|y|}\big)\Big\}\langle y_0\rangle^{-3} |y|^2;\\
	& E(y+y_0)-E(y_0)-E'(y_0)y\geq 2^{-4}\langle y_0\rangle^{-3}\, E(y).
	\end{align*}
\end{lemma}
We turn to similar convexity properties of $f$, assuming that it satisfies assumptions (H1), (H2) and (H3). 

Given $y_0\in \R^N$ we define the linearized functions $E_{y_0}$ and $f_{y_0}$ by the formula:
\begin{align*}
f_{y_0}(y)&:=f(y_0+y)-f(y_0)-f'(y_0).[y]\\
&=\int_{0}^1 (1-t) f''(y_0+ty)[y,y]\, dt.
\end{align*}
\begin{lemma}[Lemma 4.2 in \cite{Kristensen18}]\label{auxiliary2}
	For all $y,y_0\in \R^N$ with  $|y_0|\leq \alpha$, $v\in\R^m$ and $\xi\in\R^n$  we have 
	\begin{align}
	& |f_{y_0}(y)|\lesssim_\alpha L\, E(y),\label{aux1}\\
	& |f'_{y_0}(y)|\lesssim_\alpha L\, \min\{|y|,1\},\label{aux2}\\
	& |f''_{y_0}(0)y-f'_{y_0}(y)|_{\R^N}\lesssim_\alpha L\, E(y),\label{aux3}\\
    & f''_{y_0}(0)[\Bhat[\xi]v,\Bhat[\xi]v] \gtrsim_\alpha \ell |v|^2|\xi|^2.
	\end{align}
    Furthermore, for all $\varphi\in C^1_c(\rn,\R^m)$ it holds 
    \begin{equation}
	\int_\rn  f_{y_0}(\B \varphi(x))\, dx\gtrsim_\alpha \ell \int_\rn E(\B \varphi(x))\, dx.
 \end{equation}
\end{lemma}
	In particular the last inequality is the $\B$-Legendre-Hadamard ellipticity condition introduced in \eqref{eq:blegendrehadamarddef}.
\section{Proof of Theorem \ref{thm:epsreg}}\label{sec:proofepsreg}
We fix $u\in BV^\B(\Om)$ satisfying the local minimality condition
	\begin{equation}\label{eq:localminimalitycondition}
	\int_\Om f(\B u)\leq \int_\Om f(\B (u+\varphi))\text{ for all }\varphi\in BV^\B_0(\Om),
	\end{equation}
	where the lagrangian $f\colon \R^N\to \R$ satisfies \eqref{eq:flingrowth},\eqref{eq:fstrongBqc} and \eqref{eq:freg}.
\subsection{Euler-Lagrange equation}
We start with a
\begin{lemma}\label{lem:triangularlambdaconvex}
	For all $\lambda,\lambda'\in\Lambda$ we have $f^\infty(\lambda+\lambda')\leq f^\infty(\lambda)+f^\infty(\lambda').$
\end{lemma}
\begin{proof} We exploit that $f$ is $\Lambda$-convex. Fix any $t>1$ and write the two-slope inequality along the line $\{\lambda+s\lambda': s\in\R\}$
	$$
	\frac{f(t\lambda+t^2\lambda')-f(t\lambda)}{t^2}\geq \frac{f(t(\lambda+\lambda'))-f(t\lambda)}{t},
	$$
we conclude sending $t\to+\infty$, and using the existence of the strong recession function at points in $\Lambda$. \end{proof}
The following Proposition is inspired by \cite[Lemma 2.15]{Kristensen18}.
\begin{proposition}[Euler-Lagrange equation]\label{prop:euler}
	For every $\varphi \in BV^\B_0(\Om)$ we have
	\begin{equation}\label{eq:euler}
	-\int_{\Om } f^\infty(\B \varphi^s)\leq \int_\Om f'\left({\B u^{ac}(x)}\right).\left[\B\varphi^{ac}(x)\right]\, dx\leq \int_\Om f^\infty(-\B \varphi^s).
	\end{equation}
	In particular, using smooth variations we find that in $\Om$
	$$
	\B^*\left[f'\left({\B u^{ac}}\right)\right]=0\quad \text{ in the sense of distributions.}
	$$
\end{proposition}
\begin{proof}
Let $\varepsilon>0$, notice that by uniqueness of the Lebesgue decomposition of measures we have
\begin{equation*}
    (\B(u+\eps\varphi))^s = (\B u)^s + \eps (\B \varphi)^s,
\end{equation*}
thus we use the singular measure $\tau:=|\B^s u|+|\B^s\varphi|$. By Besicovitch differentiation and \eqref{eq:rankone} we have
$$
\frac{d(\B u)^s}{d\tau}(x)=\frac{d(\B u)^s }{d|(\B u)^s|}(x)\cdot\frac{d|(\B u)^s|}{d \tau}(x)\in\Lambda\quad \text{ for }\tau\text{-a.e. } x\in\Om,
$$
so that, since $f^\infty$ is positively 1-homogeneous
\begin{equation*}
    \int_\Om f(\B u) = \int_\Om f((\B u)^{ac}(x))\, dx +\int_\Om f^\infty\Big(\frac{d(\B u)^s}{d\tau}\Big) \,d\tau.
\end{equation*}
The same holds for $\varphi$ and $u+\eps\varphi$, so by \Cref{lem:triangularlambdaconvex} we find
\begin{align*}
	\int_\Om f^\infty \left(\B( u+\varepsilon  \varphi)^s \right)-f^\infty \left(\B u^s\right)=\int_\Om f^\infty\left(\frac{d\B(u+\varepsilon \varphi)^s}{d\tau }\right)-f^\infty\left(\frac{d\B u^s}{d\tau}\right)\, d\tau	\\
	\leq\int_\Om f^\infty\left(\varepsilon\frac{d\B\varphi^s}{d \tau}\right)\, d\tau =\varepsilon \int_\Om f^\infty\left(\B \varphi^s\right).
\end{align*} 
Combining this inequality with the local minimality condition \eqref{eq:localminimalitycondition} we find
\begin{align*}
    0&\leq\int_\Om f(\B(u+\varepsilon \varphi))-\int_\Om f(\B u)\\
	&=\int_\Om \int_0^1 f'\left(\B u^{ac} +t \varepsilon \B\varphi^{ac}\right)[\varepsilon \B\varphi^{ac}] \, dt\,  dx + \int_\Om f^\infty \left(\B u^s+\varepsilon \B \varphi^s \right)-f^\infty \left(\B u^s\right)\\
	&\leq \varepsilon \int_\Om \Big(\int_0^1 f'\left(\B u^{ac} +t \varepsilon \B\varphi^{ac}\right)\, dt \Big)[\B\varphi^{ac}] \, dx + \varepsilon \int_\Om f^\infty\left(\B \varphi^s\right).	
\end{align*}
Since $f$ is globally Lipschitz, sending $\varepsilon\to 0^+$ and using the dominated convergence theorem we find
$$
    \int_\Om f'\left({\B u^{ac}}\right).\left[\B \varphi^{ac}\right]\,\geq - \int_\Om f^\infty\left(\B \varphi^s\right).
$$
Using as test function $-\varphi$ we get the opposite inequality and thus \eqref{eq:euler}.
\end{proof}
\subsection{Caccioppoli inequality}
The next step consists in establishing a nonlinear Caccioppoli inequality, combining the local minimality condition, the strong $\B$-quasiconvexity, and Widman's ``hole filling" trick. This technique goes back to Evans \cite{Evans1986}.

\begin{proposition}[Caccioppoli inequality]\label{caccioppoli}
Fix a threshold $\alpha>0$. For every $a\in \aff(\B,B_R(x_0))$ with $B_R(x_0)\Subset \Om$ and $|\B a|\leq \alpha$, we have
\begin{equation}
\int_{B_{R/2}(x_0)}E(\B(u-a))\lesssim_{\alpha,n,\B,L/\ell}  \int_{B_R(x_0)}E\left(\frac{u-a}{R}\right).
\end{equation}
\end{proposition}
\begin{proof}
    Assume by simplicity $x_0=0$. Set $y_0:=\B a$, $\tilde u := u-a$ and $ f_{y_0}(y):= f(y+y_0)-f(y_0)-f'(y_0).y$.
	
	\textbf{Step 1.} For every $\varphi\in BV^\B_c(\Om)$ there holds
	\begin{equation}\label{shiftedmin}
		\int_\Om f_{y_0}(\B \tilde u+\B \varphi )\geq \int_\Om f_{y_0}(\B \tilde u).
	\end{equation}
	Just subtract to the local minimality condition 
	$$
		\int_\Om f(y_0 +\B \tilde u+\B \varphi )\geq \int_\Om f(y_0 +\B \tilde u),
	$$
 the identity
	$$
	\int_\Om f(y_0)+\int_\Om f'(y_0).[\B \varphi +\B\tilde u]=\int_\Om f(y_0) +\int_\Om f'(y_0)[\B \tilde u].
	$$
 Since $y_0\in \R^N$ lives in the ball of radius $\alpha$ the bounds of Lemmas \ref{auxiliary1} and \ref{auxiliary2} are available.
 
 Just as in the classical Caccioppoli inequality we use as test function the solution itself. We choose two balls $B_s\subset B_t$, with radii $R/2<s<t<R$, and a cutoff function $\bm 1_{B_s}\leq \eta\leq \bm 1_{B_t}$ smooth and satisfying $|\nabla \eta|\leq 2/(t-s)$. 
 
 \textbf{Step 2.} We use \eqref{shiftedmin} with $\varphi:=-\eta\tilde u$ and  \eqref{aux1}, 
	\begin{align}\label{caccupper}
	\int_{B_t} f_{y_0}(\B \tilde u) \leq \int_{B_t} f_{y_0}(\B \tilde u + \B \varphi)=\int_{B_t\setminus B_s} f_{y_0}(\B \tilde u + \B \varphi)\lesssim_\alpha L  \int_{B_t\setminus B_s}E(\B v)
	\end{align}
	where we set $v:=(1-\eta)\tilde u =\tilde u +\varphi$.
	
	\textbf{Step 3.} Let $\varrho_\varepsilon$ be a family of standard smooth mollifiers. Since $f_{y_0}$ is strongly $\B$-quasiconvexity at $y=0$ (remember that $ f_{y_0} (0)=E(0)=0$) we find
	$$
	0\leq \int_{B_t} f_{y_0}(\B(\eta(\tilde u* \varrho_\varepsilon)))-\ell E(\B(\eta(\tilde u*\varrho_\varepsilon))).
	$$
	Since $\spt \eta\subset B_t$, Lemma \ref{lem:approximationBVB} implies that $\B(\eta(\tilde u* \varrho_\varepsilon))\areastrict -\B\varphi$ in $B_t$, as $\varepsilon\to 0^+$. Thus by area-strict continuity we find
	\begin{equation}\label{cacclower}
	\ell \int_{B_s} E(\B \tilde u)\leq\ell \int_{B_t}E(\B \varphi)\leq \int_{B_t} f_{y_0}(\B\tilde u- \B v).
	\end{equation}
	\textbf{Step 4.} There exist a constant $\theta=\theta({\alpha,n,\B,L/\ell})\in(0,1)$ such that
	\begin{equation}\label{caccdecay}
		\int_{B_s}E(\B \tilde u)\leq \theta \int_{B_t} E(\B \tilde u)+\theta \int_{B_R}E\left(\frac{\tilde u}{t-s}\right).
	\end{equation}
	In order to prove this we link \eqref{caccupper} and \eqref{cacclower}, in the following way
	\begin{align*}
	\tag{by \ref{cacclower}}
		\ell \int_{B_s} E(\B \tilde u)&\leq \int_{B_t} f_{y_0}(\B\tilde u- \B v)\\
		\tag{by \ref{aux2}}
		&\lesssim_\alpha \int_{B_t} f_{y_0}(\B \tilde u)+ L \int_{B_t}E(\B v)\\
		\tag{by \ref{caccupper}}
		& \lesssim_\alpha L  \int_{B_t\setminus B_s}E(\B v)+ L \int_{B_t\setminus B_s}E(\B v)\\
		\tag{by Leibnitz}
		&\approx L\int_{B_t\setminus B_s}E\left((1-\eta)\B \tilde u +\Bhat[\nabla\eta]\tilde u\right) \\
		\tag{by \ref{quasihomogeneity} and \ref{quasitriangular}}
		&\lesssim_{\B} L \int_{B_t\setminus B_s}E(\B \tilde u) + L \int_{B_t\setminus B_s}E\left(\frac{\tilde u}{t-s}\right).
	\end{align*}
	In follows that there exist $c=c({\alpha,L/\ell,\B})$ such that
	$$
	\int_{B_s} E(\B \tilde u)\leq c\, \int_{B_t\setminus B_s}E(\B \tilde u) + c\,  \int_{B_R}E\left(\frac{\tilde u}{t-s}\right),
	$$
	next we fill the hole adding to both sides the term $c\, \int_{B_s}E(\B \tilde u)$. Setting $\theta:=\frac{c}{c+ 1 }$ we find \ref{caccdecay}.
	
	\textbf{Step 5.} The conclusion follows iterating Lemma \ref{discreteiterationlemma} below with 
	$$\varPhi(r):=\int_{B_r}E(\B \tilde u)\text{ and } \varPsi(t):=\int_{B_R}E\left({\tilde u /t}\right),$$
	once we notice that with this choice \eqref{caccdecay} is assumption \eqref{eq:interationlemmaassumption} and $\varPsi(h/2)\leq 4\varPsi(h)$ because of \eqref{quasihomogeneity}.
	\end{proof}
\begin{lemma}\label{discreteiterationlemma}
	Let $\varPhi,\varPsi:(0,R]\to \R^+$ such that $\varPhi$ is increasing, $\varPsi$ is decreasing, $\varPsi(h/2)\leq 4\ \varPsi(h)$ for every $h>0$ and 
	\begin{equation}\label{eq:interationlemmaassumption}
	    	\varPhi(s)\leq \theta\,  \varPhi(t)+\theta\, \varPsi(t-s)\text{ for all }R/2\leq s<t\leq R.
	\end{equation}
	Then we have $\varPhi(R/2)\lesssim_\theta \varPsi(R)$.
	\end{lemma}

\subsection{Linearisation}
We fix from now on an exponent $p\in(1,(n+1)/n)$, say
\begin{equation}\label{eq:pchoiche}
    p:=\frac{2n+1}{2n}.
\end{equation}
Then we have the following harmonic replacement Lemma for good spheres.

\begin{proposition}[Linearisation]\label{linearisation}
	 Fix $\alpha>0$ and $1< q<n/(n-1)$. Then
	\begin{itemize}
		\item for every $B_R(x_0)\Subset \Om$ such that $|\B u|(\de B_R)=0$ and $\de B_R$ is a {good sphere for }$u$ in the sense of \Cref{lem:goodradius},
		\item for every $a\in \aff(\B,U)$ such that $|\B a|\leq \alpha$, where ${B}_R\Subset U\subset\Om$,
	\end{itemize}
	there exist a unique $h\in W^{1,p}(B_R,\R^m)$ which solves the system
	\begin{equation}\label{lin1}
	\begin{cases}
	\B^*\left(f''(\B a).\B h\right)=0 & \text{ in the sense of distributions,}\\
	\tr_{\de B_R(x_0)}h=\tilde u & \mathcal{H}^{n-1}\text{-a.e. on }\de B_R(x_0),
	\end{cases}
	\end{equation}
    where $\tilde u$ denotes the precise representative of $u$.
	Furthermore, $h$ satisfies
	\begin{align}\label{lin2}
	\|\nabla h-\nabla a\|_{L^p(B_R,\R^m)}&\lesssim_{{\alpha,p,\B,L/\ell}} \, \left[{\tilde u-a}\right]_{W^{1-1/p,p}(\de B_R,\R^m)},\\
	\label{lin3}
	\int_{B_R} E\Big(\frac{u-h}{R}\Big)\, dx& \lesssim_{\alpha,p,q,\B,L/\ell} R^{n(1-q)} \left(\int_{B_R} E(\B (u-a))\right)^q.
\end{align}
\end{proposition}
\begin{proof}
	We set for brevity $B:=B_R(x_0)$. Set $y_0:=\B a$ and recall that by Lemma \ref{auxiliary2}
	$$
		| f ''(0).y- f'(y)|\lesssim_\alpha L\, E(y)\text{ for all }y\in \R^N.
	$$
	\textbf{Step 1.} Set $\tilde u:=u-a\in BV^\B_{loc}(U)$ and notice that, as in the first step of Proposition \ref{caccioppoli}, it satisfies the local minimality condition
	$$
	\int_U f_{y_0}(\B \tilde u)\leq \int_U f_{y_0}(\B \tilde u+\B \varphi) \text{ for all }\varphi\in BV^\B_c(U),
	$$
    so applying \Cref{prop:euler} to $\tilde u$ we find 
	\begin{equation}\label{eulerutilde}
		\B^*\left(f'_{y_0}(\B ^{ac} \tilde u(x))\right)=0\qquad  \text{ weakly in }U.
	\end{equation}
	\textbf{ Step 2.} We have
	\begin{equation}\label{linearisationupper}	
\int_B  f''_{y_0}(0)[\B \tilde u,\B \varphi(x)]\lesssim_{\alpha,L}\int_B E(\B\tilde u)\,|\B \varphi(x)|\text{ for all }\varphi\in C^\infty_c(B,\R^m).
	\end{equation} 
	Indeed, using \ref{eulerutilde}:
	\begin{align*}
			\int_B  f''_{y_0}(0)[\B \tilde u,\B \varphi(x)]&=\int_B f''_{y_0}(0)[\B^s \tilde u,\B \varphi(x)] +\int_B  f''_{y_0}(0)[\B^{ac}\tilde u(x),\B\varphi(x)]\\
		&=\int_B f''_{y_0}(0)[\B^s \tilde u,\B \varphi(x)]  \\
		&\qquad\qquad+\int_B \left( f''_{y_0}(0).\B^{ac}\tilde u(x) -f'_{y_0}(\B^{ac}\tilde u(x))\right)\cdot\B \varphi(x)\, dx\\
		\tag{by \ref{aux3}}&\lesssim_{\alpha,L}  \int_B |\B^s \tilde u||\B \varphi|+ \int_B  E(\B^{ac}\tilde u)|\B \varphi|\, dx\\
		&=\int_B E(\B\tilde u)|\B\varphi|
	\end{align*}
	\textbf{ Step 3.} Let $\tilde h \in W^{1,p}(B,\R^m)$ be the solution of\footnote{Existence and uniqueness are assured by Theorem \ref{ellipticregconstantcoeff}, part $(a)$ and the fact that $\tilde u \in W^{1-1/p,p}(\de B,\R^m)$ since $a\in C^\infty(U)$.}
	$$
	\begin{cases}
	\B^*\big( f''_{y_0}(0).\B \tilde h\big)=0 & \text{ in }B,\\
	\tilde h=\restricts{ \tilde u}{\de B} & \text{ on }\de B.
	\end{cases}$$
	In particular we have 
	\begin{equation}\label{eulerhtilde}\int_B  f''_{y_0}(0)[\B \tilde h,\B\varphi]\, dx=0 \text{ for all }\varphi \in \D(B).
	\end{equation}
	By part (b) of \Cref{ellipticregconstantcoeff} it also holds
	\begin{equation*}
			\|\nabla\tilde h\|_{L^p(B,\R^m)}\lesssim_{\alpha,\B} [\tilde u]_{W^{1-1/p,p}(\de B,\R^m)}.
	\end{equation*}
    $W^{1,p}$-extend $\tilde h$ to $U$ and set $h:=\tilde h +a$, the last inequality proves \eqref{lin2}.

 	\textbf{Step 4.} We now want to study the size of the error $v:=\tilde u -\tilde h = u-h\in BV^{\B}_\loc(U)$. Subtracting equations \ref{linearisationupper} and \ref{eulerhtilde} we get that $v$ satisfies
	\begin{equation}\label{lin4}
	\int_B  f''_{y_0}(0)[\B v,\B\varphi]\lesssim_{\alpha,L}\int_B E(\B\tilde u)|\B\varphi| \text{ for all }\varphi \in C^\infty_c(B).
	\end{equation} 
Actually, for all $\beta>0$ \eqref{lin4} holds for every $\varphi\in W^{1,\infty}_0\cap C^{1,\beta}(B,\R^m)$. In fact both sides of \eqref{lin4} are continuous under the convergence
$$
``\B\varphi_k(x)\to \B\varphi(x) \text{ and } |\B\varphi_k(x)|\leq M,\quad\text{ for every }x\in B",
$$
and we can approximate (under this convergence) any such $\varphi$ with smooth functions by extending it to zero outside $B$, mollifying and shrinking back the support inside the ball.

\textbf{Step 5.} In this step we provide a calibration field that will quickly give the last estimate \ref{lin3}. We shift everything back to the unit ball $\mathbf{B}$, put for $x\in\mathbf{B}$
	$$
	V(x):=\frac{1}{R} v(x_0+Rx),\quad \Phi(x):=\frac{1}{R}\varphi(x_0+Rx),\quad \tilde U(x):=\frac{1}{R}\tilde u(x_0+Rx).
	$$ 
	Then \ref{lin4} becomes
	\begin{equation}\label{lin4strong}
	\int_{\mathbf{B}} f''_{y_0}(0)[\B V,\B\Phi]\lesssim_{\alpha,L}\int_{\mathbf{B}} E(\B\tilde U)|\B\Phi|  \text{ for all }\Phi \in W^{1,\infty}_0\cap C^{1,\beta}(\mathbf {B},\R^m).
	\end{equation}
	We choose $\Phi=\Phi_0$ which is the solution of
	\begin{equation}\label{calibration}
	\begin{cases}
	-\B^*\left( f''_{y_0}(0).\B \Phi_0\right)=T(V) & \text{ in }\mathbf{B},\\
	\Phi_0=0 & \text{ on }\de \mathbf{B},
	\end{cases}
	\end{equation}
	here $T(v)=v$ for $|v|\leq 1$ and $T(v)=v/|v|$ for $|v|\geq 1$, is the vectorial truncation map. Since $T(V)\in L^\infty(\mathbf{B})$, Theorem \ref{ellipticregconstantcoeff} (b) and Morrey's embedding, give
\begin{equation}\label{phizeroestiamtes}
	\|\nabla \Phi_0\|_{C^\beta(\mathbf B)}\lesssim_{q} \|\Phi_0\|_{W^{2,q'}(\mathbf{B})}\lesssim_{q,\B} \|T(V)\|_{L^{q'}}\leq \left(\int_\mathbf{B} E(V)\right)^{1/q'},
	\end{equation}
    for some $\beta>0$ small. We used that $q'>n$ and that $(\nabla\Phi_0)_{\mathbf{B}}=0$.
	 In particular $\nabla \Phi_0$ has a trace on $\de \mathbf B$, it follows that we can use as test map in \ref{calibration} any $\Psi\in C^\infty(\overline {\mathbf{B}},\R^m)$ and integrate by parts:
	\begin{align*}
	\int_{\mathbf{B}} T(V) \cdot \Psi\, dx &= -\int_{\mathbf{B}} \B^*\left( f''_{y_0}(0).\B \Phi_0\right)\cdot \Psi\, dx=\\	 &=\int_\mathbf{B}  f''_{y_0}(0)[\B \Phi_0,\B\Psi]\, dx +\int_{\de \mathbf{B}} \left( f''_{y_0}(0)\B\Phi_0\right)\cdot\left( \hat{\B}[x]\Psi\right)\, d\mathcal{H}^{n-1}(x).
	\end{align*} 
\textbf{Step 6. }There holds 
\begin{equation}\label{lin5}
\int_{\mathbf{B}} \min\left\{|V|,|V|^2\right\}\, dx=\int_\mathbf{B} \tilde f''_{y_0}(0)[\B \Phi_0,\B V].
\end{equation}
	The idea is to put formally $\Psi=V$, but we need some care. We set $\Psi_\eps(x):=\frac{1}{R}(v*\varrho_\eps)(x_0+Rx)$ and for $\varepsilon$ small the convolution is well-defined on $\mathbf{B}$. Now,
	\begin{itemize}
		\item[(i)] $\restricts{(h*\varrho_\eps)}{\de B}\to \tr_{\de {B}}h=\restricts{u}{\de{B}}$ strongly in $L^{p}(\de
		{B})$, because of the trace Theorem in $W^{1,p}$;
		\item[(ii)] the sequence $\| u*\varrho_\eps \|_{W^{s,p}(\de B,\R^m)}$ is bounded because of the definition of good sphere and the choice of $\varrho$;
		\item[(iii)] $\tilde u*\varrho_\eps(x)\to \tilde u(x)$ for $\sigma$-a.e. $x\in\de B$, because a good sphere is made of Lebesgue points;
		\item[(iv)] by the previous two points and the compact embedding $W^{s,p}(\de B,\R^m)\hookrightarrow L^{p^-}(\de B,\R^m)$ we get $\tilde u*\varrho_\eps\to \tilde u$ in $L^{p^-}(\de B)$,
		\item[(v)] by points $(i)$ and $(iv)$ we get $\restricts{\Psi_\eps}{\de\mathbf{B}}\to 0$ in $L^{p^-}(\de\mathbf{B})$;
		\item[(vi)] $\B \Psi_\eps\weakstar \B V$ in $C_c((U'-x_0)/R,\R^N)^*$ and $|\B V|(\de\mathbf{B})=|\B^s u|(\de B)=0$ by assumption, so $\B \Psi_\eps\weakstar \B V$ in $C(\overline{\mathbf{B}},\R^N)^*$,
		\item[(vii)] $\Psi_\eps\to V$ in $L^p(\mathbf B)$ because of general properties of convolution.
	\end{itemize}
This observation allows us to pass to the limit
$$
\int_{\mathbf{B}} T(V) \cdot \Psi_\eps\, dx=\int_\mathbf{B} f''_{y_0}(0)[\B \Phi_0,\B\Psi_\eps]\, dx +\int_{\de \mathbf{B}} \left( f''_{y_0}(0)\B\Phi_0\right)\cdot\left( \hat{\B}[\nu]\Psi_\eps\right)\, d\sigma,
$$
obtaining
\begin{align*}
\tag{by (vii)}\lim_\eps \int_{\mathbf{B}} T(V) \cdot \Psi_\eps\, dx=\int_{\mathbf{B}} T(V) \cdot V\, dx,\\
\tag{by (vi)} \lim_\eps \int_\mathbf{B}  f''_{y_0}(0)[\B \Phi_0,\B\Psi_\eps]\, dx=\int_\mathbf{B} \tilde f''_{y_0}(0)[\B \Phi_0,\B V],\\
\tag{by (v)}
\lim_\eps \int_{\de \mathbf{B}} \left( f''_{y_0}(0)\B\Phi_0\right)\cdot\left( \hat{\B}[\nu]\Psi_\eps\right)\, d\sigma=0,
\end{align*}
so we get \ref{lin5} using $T(v)\cdot v=\min\{|v|,|v|^2\}$.

\textbf{Step 7.} We put everything together and conclude
\begin{align*}\tag{by \ref{Emin}}
\int_{\mathbf{B}} E(V)\, dx &\lesssim \int_\mathbf{B} \min\left\{|V|,|V|^2\right\}\, dx\\
	\tag{by \ref{lin5}}&= \int_{\mathbf{B}} f''_{y_0}(0)[\B \Phi_0,\B V]\\
	\tag{by \ref{lin4strong}}&\lesssim_{\alpha,L} \int_{\mathbf{B}} E(\B\tilde U)|\B\Phi_0| \\
&\lesssim_{\B} \|\nabla \Phi_0\|_{L^\infty(\mathbf B)}\int_{\mathbf{B}} E(\B\tilde U)\\
	\tag{by \ref{phizeroestiamtes}}&\lesssim_{q,\B}  \left(\int_\mathbf{B} E(V)\right)^{1/q'} \int_{\mathbf{B}} E(\B\tilde U),
	\end{align*}
	dividing we get
	$$
	\int_{\mathbf{B}} E(V)\, dx\lesssim_{\alpha,q,\B,L} \left(\int_{\mathbf{B}} E(\B\tilde U)\right)^{q},
	$$ 
	changing variables back to the ball $B_R(x_0)$ we finally find \ref{lin3}.
\end{proof}
\subsection{Excess decay and iteration}

In the previous section we derived two key inequalities for a minimizer $u$ of $\F$, with the previous notations they roughly look like
\begin{align*}
\fint_{B_{R/2}} E(\B u)\stackrel{\text{Caccioppoli}}{\lesssim} &\fint_{B_R} E\left(\frac{u}{R}\right)\, dx, \\
\fint_{B_R} E\left(\frac{u-h}{R}\right)\stackrel{\text{Linearized Euler}}{\lesssim}&\left(\fint_{B_R}E(\B u)\right)^q,
\end{align*} 
where $h$ is the $\B$-affine replacement of $u$ in $B_R$ and the averages are always taken with respect to the Lebesgue measure. The idea is to link these inequalities to obtain some nonlinear estimate that can be iterated on smaller balls. As in other regularity results we wish to control a suitable  ``excess'' function: in our set-up the right definition is
\begin{equation}
\ecc(x_0,R):=\int_{B_R(x_0)} E\left(\B u- (\B u)_{B_R(x_0)}\right),
\end{equation}
where we recall that for every non-neglegible and bounded Borel set $U$ we use the notation
$$
\left(\B u\right)_{U}:=\frac{\B u(U)}{\leb^n(U)}.
$$
When the center $x_0$ is fixed we will also use the shorthand $\ecc(x_0,R)=\ecc(R)$. Then we have

\begin{proposition}[Preliminary Decay]\label{preliminarydecay}
 		 Let $u\in BV^\B(\rn)$ be a minimizer of $\overline \F$. Fix a positive threshold $\alpha>0$, any exponent $1<q<n/(n-1)$ and any ball $B_R(x_0)\Subset \Om$ such that
 	\begin{equation}\label{small}
 	\left|\left(\B u\right)_{B_R(x_0)}\right|\leq \alpha\ \ \text{ and }\ \ 
 	\left|\B u -\left(\B u\right)_{B_R(x_0)}\right|(B_R(x_0))\leq \omega_n R^n.
 	\end{equation}
 	Then there is a large constant $c_{\text{\emph {dec}}}=c_\text{\emph{dec}}(\alpha,q,\B,L,\ell)$ such that
 	\begin{equation}\label{eq:decay}
 	\ecc(\sigma R)\leq c_\text{\emph{dec}}\left(\sigma^{n+2}+\frac{1}{\sigma^2}{\left(\frac{\ecc(R)}{\omega_nR^n}\right)}^{q-1}\right)\, \ecc(R),
 	\end{equation}
 	for any $\sigma\in (0,1/10)$.
 \end{proposition}
\begin{proof}
	In this proof we shall denote by $C$ a generic constant depending only on $\alpha,q,L/\ell,\B,n$. In particular, we will keep track of the dependence of the constants from $R$ and $\sigma$. Since we work at a fixed center we will forget about $x_0$.

	\textbf{Step 1.} We fix any linear map $A\in \R^m\otimes \R^n$ such that $\B(Ax)=(\B u)_{B_R}$. This is possible because $(\B u)_R\in \Lambda$.
 
	\textbf{Step 2.} Choose any $a_0\in \ker(\B,B_{R})\subset C^{\infty}(B_{R})$ such that
	$$
	\big[ u-Ax-a_0\big]_{W^{1-1/p,p}(B_{R/2})}\lesssim_\B    R^{\frac{n+1}{p}-n} |\B u- (\B u)_{B_R}|\left(B_{R}\right),
	$$ 
	this is possible thanks to the local Poincaré inequality \Cref{prop:weakpoincare} applied to the function $u(x)-Ax$ in the ball $B_{R}$.
 
	\textbf{Step 3.} Using Lemma \ref{lem:goodradius} on $u-Ax-a_0$ we find a radius $r^*\in \left(\frac{4}{10}R,\frac{5}{10}R\right)$ such that $\de B_{r^*}$ is a good sphere for $u-Ax-a_0$ (and thus for $u$), $|\B u|(\de B_{r^*})=0$ and 
	\begin{align*}
	\big[\ u- Ax-a_0 \big]_{W^{1-1/p,p}(\de B_{r^*})}&\lesssim_{n}{R^{-1/p}} \big[ u- Ax-a_0 \big]_{W^{1-1/p,p}(B_{R/2})}\\
    &\lesssim_{\B} R^{-n(1-1/p)}|\B u- (\B u)_R|\left(B_{R}\right).
	\end{align*} 

Setting $\tilde u(x):=u(x)-a_0(x)-Ax\in BV^\B_{loc}(B_R)$ we proved
	$$
	[{\tilde u}]_{W^{1-1/p,p}\de B_{r^*}}\lesssim_\B R^{-n(1-1/p)} |\B \tilde u|(B_R).
	$$
 
	\textbf{Step 4.} Set $a(x):=a_0(x)+Ax\in \aff(\B,B_{R})$ and notice that \[|\B a|=|\B(a_0+Ax)|=|0+(\B u)_{B_R}|\leq \alpha.\]
    We apply now the linearization procedure (Proposition \ref{linearisation}) in the ball $B_{r*}$ using $A$ as $\B$-affine map: the choice of $r^*$ ensures assumption $(i)$, the last estimate ensures $(ii)$. Thus we find a map $ h\in C^\infty\cap W^{1,p}(B_{r*},\R^m)$ which solves the system
	\begin{equation*}
	\begin{cases}
	\B^*\left(f''(\B a).\B h\right)=0 & \text{ in }B_{r^*},\\
	\tr_{\de B_{r^*}} h=\restricts{ u}{\de B_{r^*}} & \text{ on }\de B_{r^*},
	\end{cases}
	\end{equation*}
	and satisfies
	\begin{align}
	\|\nabla\tilde h\|_{L^p(B_{r*})}\lesssim_{\alpha,\B,L}[{ \tilde u}]_{W^{1-1/p,p}(\de B_{r^*})}\lesssim_\B R^{-n(1-1/p)} |\B\tilde u|(B_{R}),\label{rstar}\\
	\fint_{B_{r^*}} E\Big(\frac{\tilde u-\tilde h}{r^*}\Big)\, dx\lesssim_{\alpha,\B,L} \Big(\fint_{B_{r^*}} E(\B\tilde u)\Big)^q=\Big(\frac{\ecc(r^*)}{\leb^n(B_{r^*})}\Big)^q,\label{rstar2}
	\end{align}
	where we set for brevity $\tilde h:=h-a$.
 
	\textbf{Step 6.} Consider the affine map $H(x):=\tilde h(x_0)+\nabla\tilde h (x_0)(x-x_0)$. Then
	$$
	|\B(a+H)(x)|\leq \alpha+C_0\ \ \ \text{ for all }x\in B_{R},
	$$
	for some constant  $C_0=C_0(\B)$. We have
	\begin{align*}
	|\B (a+ H)|&\leq |\B a|+|\B H|\lesssim_\B |(\B u)_R|+ |\nabla\tilde h(x_0)|\leq \alpha+ \sup_{x'\in B_{r^*/2}}|\nabla\tilde h(x')|,
 \end{align*}
 but then Theorem \ref{ellipticregconstantcoeff} and \eqref{rstar} give
 \begin{align*}
     \sup_{x'\in B_{r^*/2}}|\nabla\tilde h(x')|&\lesssim_{\alpha,\B,\Lambda/\lambda} \Big(\fint_{B_{r^*}	} |\nabla\tilde h |^p\, dx\Big)^{1/p}\\
     &\lesssim_{\alpha,\B,\Lambda/\lambda} R^{-n/p} [u]_{W^{1-1/p,p}(\de B_{r^*})}\lesssim_\B R^{-n} |\B\tilde u|(B_{R})\lesssim 1.
	\end{align*}
 
	\textbf{Step 7.} Fix any $\sigma\in (0,1/10)$. We will prove the decay bounding $\ecc(\sigma R)$ in terms of $\ecc(R)$, linking the Caccioppoli inequality and the harmonic approximation.
    
    Exploiting the quasi-minimality property of the mean (Lemma \ref{quasiminimality}) we get
	$$
	\ecc(\sigma R)=\int_{B_{\sigma R}}E(\B u-(\B u)_{\sigma R})\leq 4\int_{B_{\sigma R}}E(\B u-\B(a+H))=4\int_{B_{\sigma R}}E(\B(\tilde u- H)).
	$$
	Now we link this estimate with the Caccioppoli inequality (applied on the ball $B_{\sigma R}$ with map $a+H\in\aff_\B(B_R)$ and threshold $\alpha+C_0$) and the triangular inequality:
	\begin{align*}
	\ecc(\sigma R)&\leq 4\int_{B_{\sigma R}}E(\B(\tilde u- H))\lesssim_\alpha \int_{B_{2\sigma R}}E\Big(\frac{\tilde u-H}{2\sigma R}\Big)\, dx\\
	&\lesssim_\alpha \underbrace{\int_{B_{2\sigma R}}E\Big(\frac{\tilde u-\tilde h}{2\sigma R}\Big)\, dx}_{:=\bm I}+\underbrace{\int_{B_{2\sigma R}}E\Big(\frac{\tilde h-H}{2\sigma R}\Big)\, dx}_{:=\bm{II}}
	\end{align*}
	Where we used that $\tilde h$ is well defined on $B_{2\sigma R}\subset B_{r^*}$. We estimate the two integrals, let us deal with the first using Lemma \ref{lem:auxiliaryboundsE} and \eqref{rstar2}
	\begin{align*}
	\bm I &= \int_{B_{2\sigma R}}E\left(\frac{\tilde u-\tilde h}{r^*}\frac{r^*}{2\sigma R}\right)\, dx\leq \left(\frac{r^*}{2\sigma R}\right)^2\int_{B_{2\sigma R}}E\left(\frac{\tilde u-\tilde h}{r^*}\right)\, dx \\
	&\lesssim \frac{1}{\sigma^2}\int_{B_{r^*}}E\left(\frac{\tilde u-\tilde h}{r^*}\right)\, dx\lesssim_{\alpha,\B} \frac{{r^*}^n}{\sigma^2}\, {\left(\frac{\ecc(r^*)}{\leb^n(B_{r^*})}\right)}^q \lesssim\frac{\ecc(R)}{\sigma^2}\,  {\left(\frac{\ecc(R)}{R^n}\right)}^{q-1}.
	\end{align*}
	For the second term we use that $\tilde h$ behaves like an harmonic function. Using Taylor's theorem and Theorem \ref{ellipticregconstantcoeff} $(b)$ we have
 \begin{equation*}
     \sup_{B_{2\sigma R}}{|\tilde h-H|}\lesssim_{\alpha,\B} \sigma^2R^2 \sup_{B_{r^*/2}}|\nabla^2 \tilde h|\lesssim_{\alpha,\B,\Lambda/\lambda} \sigma^2 R\Big(\fint_{B_{r^*}} |\nabla\tilde h |^p\, dx\Big)^{1/p},
 \end{equation*}
 then using \eqref{rstar} as in Step 6 we find
 \begin{equation*}
     \Big(\fint_{B_{r^*}} |\nabla\tilde h |^p\, dx\Big)^{1/p}\lesssim_{\alpha,\B,\Lambda/\lambda} R^{-n} |\B\tilde u|(B_R).
 \end{equation*}
    Integrating over $B_{2\sigma R}$ we get exploiting Lemma \ref{lem:auxiliaryboundsE} and Jensen's inequality
	\begin{align*}
	\bm{II}&=\int_{B_{2\sigma R}}E\Big(\frac{\tilde h-H}{2\sigma R}\Big)\, dx\le \sigma^n R^n E\Big(C  \sigma	 \frac{|\B\tilde u|(B_R)}{\leb^n(B_R)}\Big)\\
	&\le C^2 \sigma^{n+2} R^n \left(\frac{|\B\tilde u|(B_R)}{\leb^n(B_R)}\right)^2
	\lesssim \sigma^{n+2} R^n E\Big(\frac{|\B\tilde u|(B_R)}{\leb^n(B_R)}\Big)\\
	&\leq \sigma^{n+2} R^n \fint_{B_R}E(\B\tilde u)\lesssim\sigma^{n+2} \ecc(R).
	\end{align*}
	Combining these two estimates we have \eqref{eq:decay}.
\end{proof}
The next step is to iterate this decay to prove that there is a ``critical threshold $\eps_0$'' such that, if $\ecc(x_0,R)\leq\eps_0 R^n$, then $\ecc(x_0,r)\lesssim r^{n+\gamma}$ when $	r\to 0^+$. Since this estimate will hold also for $x$ near $x_0$ too, we will be able to employ Campanato's integral characterization of H\"older continuity. We prove a decay of the normalized excess of a ball $B_R\Subset \Om$:
$$
\necc(x_0,R):=\frac{\ecc(x_0,R)}{\leb^n(B_R(x_0))}=\fint_{B_R(x_0)}E\left(\B u -(\B u)_{B_R(x_0)}\right).
$$
\begin{proposition}[Excess decay]\label{excessdecay}
	Let $u\in BV^\B(\rn)$ be a minimizer of $\overline \F$ and $\gamma\in (0,1)$ a fixed exponent. Then for every $\alpha>0$ there exists a critical threshold $\eps_{\text{\emph{crit}}}=\eps_{\text{\emph{crit}}}(\alpha,\gamma,\B,L,\ell)>0$ such that the following implication holds: if
	\begin{equation}
	 B_R(x_0)\Subset \Om,\quad	\left|(\B u)_{B_R(x_0)}\right|\leq \alpha ,\quad \necc(x_0,R)\leq \eps_{\text{crit}},
	\end{equation}
	then
	\begin{equation}\label{decay}
	 \necc(x_0,r)\lesssim_{{\alpha,\gamma,\B,L,\ell}} \left(\frac{r}{R}\right)^{2\gamma} \necc(x_0,R)\quad \text{ for every }r\in(0,R).
	\end{equation}
\end{proposition}
\begin{proof}
	Fix $u,q,\alpha$ and $\gamma$ as in the hypothesis. Let us denote with 
	$$C:=c_\text{decay}(\alpha+1,q:=1+1/n,\B,L,\ell)$$ 
	the constant given by Proposition \ref{preliminarydecay} relative to $\alpha+1$, this ensures
	\begin{equation*}
	\begin{cases}
	 B_R(x_0)\Subset \Om,	\left|(\B u)_{B_R(x_0)}\right|\leq \alpha +1, \\
	 \fint_{B_R(x_0)}\left|\B u -(\B u)_{B_R(x_0)}\right|\leq 1
	\end{cases}\Rightarrow 	\necc(x_0,\sigma R)\leq C\, \left(\sigma^{n+2}+ \sigma^{-2}\necc(x_0,R)^{1/n}\right)\necc(x_0,R),
	\end{equation*}
	for every $ 0<\sigma<1/10$. \\
	Step 1. We now fix two values $\sigma_0$ and $\eps_0$, both depending only on $\{\alpha+1,q,\gamma,\B,L,\ell\}$ such that the following holds:
	\begin{equation}\label{firstepdecay}
	\begin{cases}
	B_R(x_0)\Subset \Om,	\left|(\B u)_{B_R(x_0)}\right|\leq \alpha +1,\\
	\necc(x_0,R)\leq \eps_0
	\end{cases}\Rightarrow 	\necc(x_0,\sigma_0 R)\leq  \sigma_0^{2\gamma}\necc(x_0,R).
	\end{equation}
	The choices are
	$$
	\sigma_0:=\min\left\{\frac{1}{20},3C^{-\frac{1}{n-2(1-\gamma)}}\right\},\quad \eps_0:=\min\left\{\frac{1}{3},\left(\frac{\sigma_0^{2(1+\gamma)}}{3C}\right)^{n}\right\},
	$$
	in order to reduce \ref{firstepdecay} to Proposition \ref{preliminarydecay} just use Lemma \ref{sqaureroot}:
	$$
	\necc(x_0,R)\leq \eps_0\leq1/3\Rightarrow \fint_{B_R(x_0)}\left|\B u - (\B u)_{B_R(x_0)}\right|\leq\sqrt{3\necc(x_0,R)}\leq 1.
	$$
	We remark that $\eps_{\text{crit}}$ is yet to be chosen and differs from $\eps_0$.\\  
	Step 2. We inspect how the hypothesis of \ref{firstepdecay} behaves when passing from $B_R$ to $B_{\sigma_0R}$, that is to say we notice that 
	\begin{equation}
	\begin{cases}
	B_R(x_0)\Subset \Om, \necc(x_0,R)\leq \eps_0, \\
		\left|(\B u)_{B_R(x_0)}\right|\leq \alpha<\alpha +1
	\end{cases}
	\Rightarrow 	
	\begin{cases}
	B_{\sigma_0 R}(x_0)\Subset \Om,	\necc(x_0,\sigma_0R)\leq \sigma_0^{2\gamma}\necc(x_0,R)\leq \eps_0, \\
	\left|(\B u)_{B_{\sigma_0 R}(x_0)}\right|\leq \alpha +\sigma^{-n}_0\sqrt{3\necc(x_0,R)},
	\end{cases}
	\end{equation}
	the only nontrivial verification being
	\begin{align*}
		\left|(\B u)_{B_{\sigma_0 R}(x_0)}\right|&\leq \left|(\B u)_{B_R(x_0)}\right|+\left|(\B u)_{B_{\sigma_0R}(x_0)}-(\B u)_{B_R(x_0)}\right|\tag{triangular}\\
		&\leq \alpha + \sigma^{-n}_0\fint_{B_R(x_0)}\left|\B u - (\B u)_{B_R(x_0)}\right|\tag{monotonicity}\\
		&\leq \alpha +\sigma^{-n}_0\sqrt{3\necc(x_0,R)}\tag{ by Lemma \ref{sqaureroot}}.
	\end{align*}
	In order to apply again \ref{firstepdecay} on $B_{\sigma_0R}(x_0)$ we need an additional smallness condition on the initial excess 
\begin{equation}\label{secondstepdecay}
 \sigma^{-n}_0\sqrt{3\necc(x_0,R)}\leq 1.
\end{equation}
This is of course possible replacing $\eps_0$ with some smaller $\eps_1\leq\eps_0$. Nevertheless we want to iterate this argument infinite times so we must ensure that $\eps_k$ does not go down to $0$. Luckily this can be done, we think it is more clear to provide the next step of the iteration instead of a formal induction.\\
Step 3. Suppose that \ref{secondstepdecay} is true, then we apply \ref{firstepdecay} on $B_{\sigma_0R}(x_0)$ and get
	\begin{align*} 	
&\begin{cases}
B_{\sigma_0 R}(x_0)\Subset \Om,	\necc(x_0,\sigma_0R)\leq \sigma_0^{2\gamma}\necc(x_0,R)\leq \eps_0, \\
\left|(\B u)_{B_{\sigma_0 R}(x_0)}\right|\leq \alpha +\sigma^{-n}_0\sqrt{3\necc(x_0,R)}\leq \alpha + 1,
\end{cases}\\
\Rightarrow
&\begin{cases}
B_{\sigma_0^2R}(x_0)\Subset \Om, \necc(x_0,\sigma_0^2 R)\leq \sigma_0^{2\gamma}\necc(x_0,\sigma_0 R), \\
\left|(\B u)_{B_{\sigma_0^2 R}(x_0)}\right|\leq \alpha +\sigma^{-n}_0\sqrt{3\necc(x_0,R)}+\sigma^{-n}_0\sqrt{3\necc(x_0,\sigma_0 R)},
\end{cases}\\
\Rightarrow
&\begin{cases}
B_{\sigma_0^2R}(x_0)\Subset \Om, \necc(x_0,\sigma_0^2 R)\leq \sigma_0^{4\gamma}\necc(x_0,R), \\
\left|(\B u)_{B_{\sigma_0^2 R}(x_0)}\right|\leq \alpha +(1+\sigma_0^\gamma)\sigma^{-n}_0\sqrt{3\necc(x_0,R)},
\end{cases}
\end{align*}
In order to apply \ref{firstepdecay} on $B_{\sigma_0^2R}(x_0)$ we need a further smallness assumption on the excess
\begin{equation*}\label{thirdstepdecay}
(1+\sigma_0^\gamma)\sigma^{-n}_0\sqrt{3\necc(x_0,R)}\leq 1,
\end{equation*}
this requirement is slightly stronger than \ref{secondstepdecay}. Going on in this fashion one easily devise the pattern: the condition
\begin{equation*}\label{laststepdecay}
\left(\sum_{k\in \N} \sigma_0^{k\gamma}\right)\sigma^{-n}_0\sqrt{3\necc(x_0,R)}\leq 1,
\end{equation*}
is enough for infinitely many steps and, since the series converges, we can set
$$
\eps_{\text{crit}}:=\min\left\{\eps_0;\frac{1}{3}\sigma_0^{2n}\left(\sum_{k\in \N} \sigma_0^{k\gamma}\right)^{-2}\right\}.
$$  
With this choice we have
\begin{equation}\label{discretedecay}
\begin{cases}
B_R(x_0)\Subset \Om,	\left|(\B u)_{B_R(x_0)}\right|\leq \alpha,\\
\necc(x_0,R)\leq \eps_{\text{crit}}
\end{cases}\Rightarrow 	\necc(x_0,\sigma_0^k R)\leq  \sigma_0^{2\gamma k}\necc(x_0,R),\quad k\in\N,
\end{equation}
and this gives the H\"older estimate \ref{decay} by discrete interpolation.
\end{proof}
We can finally prove
\begin{theorem}[Main Theorem]\label{maintheorem}
	Let $u\in BV^\B(\rn)$ be a minimizer of $\overline \F$ and $\gamma\in (0,1)$ some fixed exponent. Then for every $\alpha>0$ there exists a critical threshold $\eps=\eps(\alpha,\gamma,\B,L,\ell)>0$ such that the following implication holds: if
	\begin{equation}
	B_R(x_0)\Subset \Om,\quad	\left|(\B u)_{B_R(x_0)}\right|\leq \alpha ,\quad \necc(x_0,R)\leq \eps,
	\end{equation}
	then $\B u\mres B_{R/2}(x_0) \ll \leb^n$, $\B u \in C^{0,\gamma}(B_{R/2}(x_0))$ and  
	\begin{equation}\label{holderreg}
\left[\B u \right]_{C^{0,\gamma}(B_{R/2}(x_0))}\lesssim_{\alpha,\B,L,\ell} R^{-\gamma} \sqrt{\necc(x_0,R)}.
	\end{equation}
\end{theorem}
\begin{proof}
	We start with two simple estimates that relates the oscillation of nested balls, whose centers do not necessarily agree. Given $B_R(x_0)\Subset \Om$ there holds
	\begin{align*}
		&\necc(x,R/2)\leq 2^{n+2}\, \necc(x_0,R)&\text{ for every }x\in B_{R/2}(x_0),\\
		&\left|(\B u)_{B_{R/2}(x)}\right|\leq \left|(\B u)_{B_R(x_0)}\right|+2^n \sqrt{3\necc(x_0,R)}&\text{ for every }x\in B_{R/2}(x_0).
	\end{align*} 
	These are simply proven by Lemmas \ref{quasiminimality}, \ref{sqaureroot} and the quasi-triangular inequality. So defining $\eps$ as 
	\begin{equation}
		\eps:=\min\left\{\frac{1}{3\cdot 2^n};\frac{\eps_{\text{crit}}(\alpha+1,\gamma,\B,L,\ell)}{2^{n+2}}\right\},
	\end{equation} 
	these two estimates gives us
	\begin{equation*}
	\begin{cases}
	B_R(x_0)\Subset \Om, \necc(x_0,R)\leq \eps, \\
	\left|(\B u)_{B_R(x_0)}\right|\leq \alpha
	\end{cases}
	\Rightarrow 	
	\begin{cases}
	B_{R/2}(x)\Subset \Om,	\necc(x,R/2)\leq \eps_{\text{crit}}(\alpha+1,\gamma,\B,L,\ell), \\
	\left|(\B u)_{B_{R/2}(x)}\right|\leq \alpha +1,
	\end{cases}
	\end{equation*}
	so that we can apply Proposition \ref{excessdecay} on the ball $B_{R/2}(x)$ with exponent $\gamma$ and threshold $\alpha+1$ and find:
	\begin{equation*}
		\necc(x,r)\lesssim_{{\alpha,\gamma,\B,L,\ell}} \left(\frac{r}{R}\right)^{2\gamma}\, \necc(x_0,R)\quad \text{ for every }r\in(0,R/2),
	\end{equation*} 
	but the constant involved does not depend on the center $x\in B_{R/2}(x_0)$ so there holds
	\begin{equation*}
	\necc(x,r)\lesssim_{{\alpha,\gamma,\B,L,\ell}} \left(\frac{r}{R}\right)^{2\gamma}\, \necc(x_0,R)\quad \text{ for every }x\in B_{R/2}(x_0),r\in(0,R/2).
	\end{equation*} 
	This in particular gives us that for every $x\in B_{R/2}(x_0)$ we have
	\begin{equation*}
		\limsup_{r\to 0^+} \frac{|\B^s u|(\overline {B}_r(x))}{r^n}\lesssim \limsup_{r\to 0^+}  \left(\frac{r}{R}\right)^{2\gamma}\, \necc(x_0,R)=0,
	\end{equation*}
	so by standard results about upper densities of measures (see Theorem 2.56 in \cite{ambrosio2000functions}) we get $\B u\ll \leb^n\mres B_{R/2}(x_0)$, so we can write $\B u=\B u(x)\in L^1(B_{R/2}(x_0),\R^N)$. We now show that $\B u$ belongs to the Campanato space $\mathcal{L}^{n+\gamma}(B_{R/2}(x_0))$, in fact by Lemma \ref{sqaureroot} we have
	\begin{align*}
		\left(\fint_{B_r(x)}\left|\B u(x')-(\B u)_{B_r(x)}\right|\, dx'\right)^{2}&\leq \necc(x,r)^2+2\necc(x,r)\\
		&\lesssim_{{\alpha,\gamma,\B,L,\ell}} \left(\frac{r}{R}\right)^{4\gamma}\necc(x_0,R)^2+\left(\frac{r}{R}\right)^{2\gamma}\necc(x_0,R)\\
		&\lesssim \left(\frac{r}{R}\right)^{2\gamma}\necc(x_0,R).
	\end{align*}
	We conclude that
	$$
	\sup_{x\in B_{R/2}(x_0),0<r<R/2}r^{-n-\gamma}\int_{B_r(x)}\left|\B u(x')-(\B u)_{B_r(x)}\right|\, dx'\lesssim_{{\alpha,\gamma,\B,L,\ell}}\frac{\sqrt{\necc(x_0,R)}}{R^\gamma},
	$$
	so by the integral characterization of H\"older continuity we get \ref{holderreg}:
	$$
	\left[\B u \right]_{C^{0,\gamma}(B_{R/2}(x_0))}\sim_{n,\gamma} \left[\B u \right]_{\mathcal{L}^{1,n+\gamma}(B_{R/2}(x_0))} \lesssim_{\alpha,\B,L,\ell} R^{-\gamma} \sqrt{\necc(x_0,R)}.
	$$
\end{proof}
This is the key result, using the boundedness of Calderón-Zygmund kernels between H\"older spaces one easily gets that $u\in C^{1,\gamma}(B_{R/4}(x_0),\R^m)$ for every $\gamma\in (0,1)$. Then exploiting the finite-difference method one can prove full $C^{2,\gamma}$ regularity, we do not enter in too much detail here, a quick account is given in Theorem 4.9 in \cite{Kristensen18}.

\printbibliography[heading=bibintoc]

\end{document}